\documentclass[12pt]{amsart}
\usepackage[pdftex,colorlinks,citecolor=blue,urlcolor=blue]{hyperref}
\usepackage{amssymb}
\usepackage{graphicx,color}
\usepackage{amsmath}
\usepackage{comment}
\usepackage{bbm}
\usepackage{amssymb}
\usepackage{cleveref}
\usepackage{enumerate}
\usepackage{mathtools}
\usepackage{caption}
\usepackage{subcaption}
\usepackage{tikz}
\usepackage{graphicx}
\usepackage[margin=0.25in]{geometry}
\usepackage{pgfplots}
\usepackage{pgfplotstable}
\pgfplotsset{width=10cm,compat=1.9} 
\usepackage{subcaption}
\usepgfplotslibrary{external}
\usetikzlibrary{decorations.pathreplacing,calc,plotmarks}
\newtheorem{theorem}{Theorem}[section]

\newtheorem{definition}[theorem]{Definition}

\newtheorem{question}[theorem]{Question}

\newtheorem{remark}[theorem]{Remark}

\newtheorem{observation}[theorem]{Observation}

\DeclareMathOperator{\supp}{\mathrm{supp}}

\newcommand{\N}{\mathbb{N}} 
\newcommand{\Z}{\mathbb{Z}}
\newcommand{\C}{\mathcal{C}}
\newcommand{\R}{\mathbb{R}}

\newcommand{\conv}{\mathop{\rm conv}\nolimits}

\newcommand{\ESP}{\mathrm{ESP}}

\newcommand{\TECSP}{\mathrm{TECSP}}
\renewcommand{\P}{\mathcal{P}}

\newcommand{\M}{\mathcal{M}}

\newcommand{\F}{\mathcal{F}}

\newcommand{\aff}{\mathrm{aff}}
\newcommand{\CP}{\mathrm{CP}}
\newcommand{\Cyc}{\mathrm{Cyc}}

\title{Two-edge-connected (not necessarily spanning) subgraphs and polyhedra}
\author{Justus Bruckamp, Markus Chimani, Martina Juhnke}
\address{Faculty of Mathematics/Computer Science, University of Osnabr\"{uc}k, Germany}
\email{justus.bruckamp@uni-osnabrueck.de}
\email{markus.chimani@uni-osnabrueck.de}
\email{martina.juhnke@uni-osnabrueck.de}

\begin{document}

\begin{abstract}
Given a graph $G$, we study the $2$-edge-connected subgraph polytope $\TECSP(G)$, which is given by the convex hull of the incidence vectors of all $2$-edge-connected subgraphs of $G$. We describe the lattice points of this polytope by linear inequalities which provides an ILP-algorithm for finding a $2$-edge-connected subgraph of maximum weight. Furthermore, we characterize when these inequalities define facets of $\TECSP(G)$. We also consider further types of supporting hyperplanes of $\TECSP(G)$ and study when they are facet-defining. Finally, we investigate the efficiency of our considered inequalities practically on some classes of graphs.
\end{abstract}

\maketitle
\section{Introduction}

% \cite{Steiner2edge} Steiner 2-edge survivable network, problem: S subseteq V, finding a 2-edge-connected subgraph spanning S polytope, main result: describe polytope for series parallel graphs

% \cite{MahjoubSpanning} Two-edge connected spanning subgraphs and polyhedra, problem: finding a spanning 2-edge-connected subgraph, results: studied facet defining inequalities and describe polytope for series parallel graphs

% \cite{BARAHONA199519} On two-connected subgraph polytopes, problem: 2-edge and 2-node connected spanning subgraph polytope, results: study further inequalities for 2-edge-connected spanning subgraph polytope and describe polytope for Halin graphs

% \cite{DidiBihaConnectedSubgraph} Polyhedral study of the connected subgraph problem, problem: connected subgraph polytope, result: inequalities for this polytope and complete description for cycles and trees

% Steiner tree polytope?

% \cite{EhrgottkCardinalityConnSubgraph} k-cardinality subgraphs: find connected subgraph with k edges ILP-algorithm 

% \cite{GroetschelMonmaStoer1995a} survey on survivable networks

% \cite{connectedblocks}

In network planning, it is a typical problem to ask for a weight-optimal subgraph $H$ of a given weighted graph $G$, where $H$ should have some required properties. Often, one wants to find a network that can withstand the loss of connections, e.g., $H$ should be $2$-edge- or $2$-vertex-, or even stronger connected. In this article we focus on the $2$-edge-connected case, i.e., we want to identify a weight-optimal $2$-edge-connected subgraph of a given weighted graph~$G$. Let $V(G)$ and $E(G)$ denote the vertices and edges of $G$, respectively. $G$ is \mbox{$2$-edge-connected}, if for every pair of distinct vertices $u,v \in V(G)$, there exist at least two edge-disjoint paths connecting $u$ and~$v$. We consider the \emph{2-edge-connected subgraph polytope} $\TECSP(G)$, which is given by the convex hull of all incidence vectors of $2$-edge-connected subgraphs of $G$, i.e.,
\begin{align*}
    \TECSP(G)\coloneqq \conv\{\chi^H \mid H \subseteq G \text{ is $2$-edge-connected}\}\subseteq \R^{E(G)},
\end{align*}
where $\chi^H \in \{0,1\}^{E(G)}$ is defined by
\begin{align*}
        \chi_e^H\coloneqq \begin{cases} 1, \text{ if } e \in E(H), \\
                               0, \text{ otherwise}.\end{cases}
\end{align*}
We will describe all lattice points of this polytope by linear inequalities, which provides an integer linear program (ILP) for this problem. The polytope is related to other previously studied polytopes from graph optimization problems and we also provide new ILPs for such problems.

In \cite{MahjoubSpanning, BARAHONA199519}, the problem of finding a \emph{spanning} $2$-edge-connected subgraph has been considered by studying the corresponding polytope. In both articles, several supporting inequalities have been studied. Furthermore, the polytope is completely described for series-parallel graphs in \cite{MahjoubSpanning}, and for Halin graphs in \cite{BARAHONA199519}. Given some vertex subset $R \subseteq V(G)$, \cite{Steiner2edge} and \cite{Mahjoub2008} consider the more general polytope of subgraphs $H \subseteq G$ that allow a $2$-edge-connected subgraph $H' \subseteq H$ with $R \subseteq V(H')$. Observe that $H$ itself does not even need to be connected. Again, the polytope is completely described for series-parallel graphs in \cite{Steiner2edge} and for a generalization of Halin graphs in \cite{Mahjoub2008}. Since we will describe all lattice points of $\TECSP(G)$, we can directly describe such $2$-edge-connected subgraphs $H'$ by intersecting $\TECSP(G)$ with the inequalities $\sum_{e \in \delta(\{v\})}x_e \ge 2$ for $v \in R$, where $\delta(\{v\})$ denotes the edges incident to $v$.

In \cite{DidiBihaConnectedSubgraph}, the polytope of connected subgraphs of a given graph $G$ is considered. The authors study supporting hyperplanes and completely describe the polytope if $G$ is a tree or a cycle. Therefore, studying $2$-edge-connected subgraphs and the related polytope is a natural follow-up question.

Another related problem is the \emph{$k$-Cardinality Subgraph Problem}, which asks for a weight-optimal connected subgraph that has exactly $k$ edges. In \cite{EhrgottkCardinalityConnSubgraph}, the related polytope is studied. Supporting hyperplanes of the polytope are studied and two different algorithms for this problem are compared. As for the connected subgraph problem, a natural follow-up question is to study the problem of finding $k$-cardinality $2$-edge-connected subgraphs: this can be seen as the intersection of $\TECSP(G)$ with the equality $\sum_{e \in E(G)}x_e=k$.

In addition, there are also connections to the problem of finding a weight-optimal \emph{Eulerian} subgraph, i.e., a connected subgraph in which each vertex has even degree. Since every Eulerian subgraph is $2$-edge-connected, the corresponding Eulerian subgraph polytope is contained in $\TECSP(G)$. In literature, the Eulerian subgraph polytope often appears as a special case of the cycle polytope of a matroid, if the underlying matroid is graphic, see for example \cite{BARAHONA198640}. But for the cycle polytope of a graphic matroid, all (possibly disconnected!) subgraphs are considered as long as every vertex has even degree. Therefore, this is not the actual Eulerian subgraph polytope, since also disconnected subgraphs are allowed. In \cite{BARAHONA198640}, a complete facet description is provided in the case that the underlying matroid is graphic. By this, we get an ILP for the problem of finding a weight-optimal (connected) Eulerian subgraph by combining the ILP of the corresponding cycle polytope for a graphic matroid with the ILP we provide for finding a weight-optimal $2$-edge-connected subgraph.

Furthermore, the $2$-edge-connected subgraph polytope has been studied in \cite{connectedblocks} in the special case that the underlying graph is a cactus graph, i.e., a graph where any two cycles intersect in at most one vertex. For this case, a complete facet description of the polytope is provided.

We close this section by specifically discussing bridges, i.e., edges whose deletion increase the number of connected components. Clearly, they are never contained in a $2$-edge-connected subgraph of $G$. One can delete all bridges in $G$ without changing the set of $2$-edge-connected subgraphs. Whenever $G$ contains several connected components, one can solve the problem of finding a weight-optimal $2$-edge-connected subgraph on each component independently and choose the overall best one. Therefore, we may safely assume that the underlying graph $G$ is itself $2$-edge-connected most of the time.

\smallskip

\paragraph{\textbf{Organization}} In \Cref{preliminaries}, we give the necessary background and notation used in this paper. In \Cref{facets}, we discuss five different types of feasible inequalities for $\TECSP(G)$, and characterize when they define facets. Furthermore, in \Cref{lattice points}, we describe all lattice points of $\TECSP(G)$ by linear inequalities, which yields an ILP for finding a weight-optimal $2$-edge-connected subgraph. In \Cref{experiments}, we discuss separability for the different considered inequalities and report on experiments investigating the advantages and disadvantages of using additional facet-defining inequalities in practice. \Cref{conclusion} concludes with some open questions.

\section{Preliminaries}

\label{preliminaries}

For $n \in \N$, let $[n] \coloneqq \{1,\ldots,n\}$ and for a set $M$, let $2^M$ denote the power set of $M$, and let $\binom{M}{2}$ denote the set of all subsets of $M$ of cardinality $2$. For vectors $v_1,\ldots,v_k \in \R^n$ let $\mathrm{span}(v_1,\ldots,v_k)$, $\aff(v_1,\ldots,v_k)$, and $\conv(v_1,\ldots,v_k)$ denote their span, their affine hull, and their convex hull, respectively.
Let $G$ be a simple (i.e., no multi-edges) undirected graph with vertex set $V(G)$ and edge set $E(G) \subseteq \binom{V(G)}{2}$. For a subset of the vertices $S \subseteq V(G)$, let $G[S]\coloneqq (S,E(G) \cap \binom{S}{2})$ denote the graph induced by $S$.

Let $F \subseteq E(G)$ be a subset of the edges of $G$. We define the \textit{incidence vector} $\chi^{F} \in \{0,1\}^{E(G)}$ of $F$ by
    \begin{align}
    \label{incidence vector}
        \chi_e^F\coloneqq \begin{cases} 1, \text{ if } e \in F, \\
                               0, \text{ otherwise}.\end{cases}
    \end{align}

For a subgraph $H \subseteq G$, we use the shorthand $\chi^H \coloneqq \chi^{E(H)}$.
For a vector $x \in \R^{E(G)}$ we will use the notation $x(F) \coloneqq \sum_{e \in F}x_e$ for $F \subseteq E(G)$. Furthermore, we denote the \emph{support of $x$} by $\supp(x) \coloneqq \{e \in E(G) \mid x_e \neq 0\}$.
For $F \subseteq E(G)$, let $G-F \coloneqq (V(G),E(G)\setminus F)$ be the graph after removing $F$. For $e \in E$ we may write $G-e$ instead of $G-\{e\}$. Similar, for $S \subseteq V(G)$ we denote by $G-S$ the graph with vertex set $V(G) \setminus S$ and edge set $E(G) \cap \binom{V(G) \setminus S}{2}$. For a single vertex $v \in V(G)$ we will also write $G-v$ instead of $G-\{v\}$.
For $\emptyset \subsetneq S \subsetneq V(G)$, let $\delta(S)$ denote all edges that have exactly one vertex in $S$, i.e.,
\begin{align*}
    \delta(S)\coloneqq \{e \in E(G)\mid \vert e \cap S\vert =1\}.
\end{align*}
The set $\delta(S)$ is also called a \textit{cut}, since it separates the vertices of $S$ from the vertices of $V(G) \setminus S$, and therefore, the graph $G-\delta(S)$ is disconnected. If $\vert\delta(S)\vert=k$, we say, that $\delta(S)$ is a \textit{cut of size $k$} or shorter \textit{$k$-cut}. If $\delta(S)=\{e\}$ for some $\emptyset \subsetneq S \subsetneq V(G)$, we say that $e$ is a \textit{bridge}. It is well-known that if $F_1,F_2 \subseteq E(G)$ are cuts, then also their symmetric difference $F_1 \triangle F_2 \coloneqq (F_1 \setminus F_2) \cup (F_2 \setminus F_1)$ is a cut. We say that a cut $F \subseteq E(G)$ has a \emph{chord} $e \in E(G)$, if there are two cuts $F_1,F_2 \subseteq E(G)$, such that $F_1 \triangle F_2=F$ and $F_1 \cap F_2 = \{e\}$. Furthermore, it is well-known that the intersection of a cut and the edge-set of a cycle has always even cardinality.

\begin{definition}
A graph $G$ is \emph{$2$-edge-connected}, if and only if one of the following equivalent conditions is true:
\begin{enumerate}[(i)]
    \item for every pair of distinct vertices $u,v \in V(G)$, there exist at least two edge-disjoint paths connecting $u$ and $v$;
    \item $G$ is connected and does not contain any bridges.
\end{enumerate}
\end{definition}

We now define the central object of this paper. For a simple undirected graph $G$, the \textit{$2$-edge-connected subgraph polytope} $\TECSP(G)$ is given by
\begin{align*}
\TECSP(G)=\conv\{\chi^H \mid H \subseteq G \text{ is $2$-edge-connected}\}.
\end{align*}
Since the empty graph, and a graph with only one vertex and no edge, is $2$-edge-connected by definition, the all-zeros vector $\textbf{0}$ is always contained in $\TECSP(G)$.

%In \cite{BARAHONA198640}, the elements of $\C$ are called the \emph{circuits} of $\M$ and a disjoint union of circuits is called a \emph{cycle}.  Given a graph $G$, one can associate a matroid to $G$ in the following way. Let $G=(V(G),E(G))$ be a graph. Let $\M(G)=(E(G),\C(G))$, where for $C \subseteq E(G)$ we have $C \in \C(G)$ if and only if there is a simple cycle in $G$ that has edges $C$. Then $\M(G)$ is a matroid and a matroid that is ismorphic to $M(G')$ for a graph $G'$ is called \emph{graphic}. It is well known that every graphic matroid is binary. 

%In one proof we will refer to \cite{BARAHONA198640}, in which the \emph{cycle polytope} $\Cyc(\M)$ of a binary matroid $\M$ has been studied. Therefore, we will use some matroid related terminology. Since we will not need most of the used matroid theory in the rest of the paper, we will not describe all of the needed background here, but in the appendix in \Cref{appendix}, where we also define the cycle polytope of a binary matroid. Nevertheless, the following matroid related background is needed and we will explain more on that here, but especially for a graphic matroid, which will be relevant in this article.

\subsection{Coparallel classes}

In \cite{BARAHONA198640}, the \emph{cycle polytope} $\Cyc(\M)$ of a binary matroid $\M$ has been studied, and in particular, \emph{coparallel elements} of a matroid $\M$ play an important role in this study. For the $2$-edge-connected subgraph polytope $\TECSP(G)$ of a graph $G$, this will also be the case, where we talk about coparallel elements in the corresponding graphic matroid. Since we only need this special case in this article, we will only define it for graphs.

Two distinct edges $e,f \in E(G)$ are \emph{coparallel} if and only if they form a minimal $2$-cut of~$G$. A \textit{coparallel class} of $G$ is an inclusion-wise maximal subset $C \subseteq E(G)$, such that every pair of distinct elements of $C$ is coparallel. Note that if an element $e \in E(G)$ is neither a bridge, nor contained in a $2$-cut, then $\{e\}$ forms a coparallel class by itself. If $e \in E(G)$ is a bridge, we do not consider $\{e\}$ to be a coparallel class. Since coparallelism defines a transitive relation, the \emph{set of coparallel classes of $G$}, denoted by $\CP(G)$, is a partition of the edge set $E(G)$, as long as $G$ does not contain any bridges, which is especially true if $G$ is $2$-edge-connected.
%Let $\M(G)$ denote the \emph{graphic matroid} of a graph $G$. In a graphic matroid two distinct edges $e,f \in E(G)$ are coparallel if and only if they form a minimal $2$-cut of $G$. A \textit{coparallel class} of $\M$ is an inclusion-wise maximal subset $C \subseteq E$ of the ground set $E$ of $\M$, such that every pair of distinct elements of $C$ is coparallel.\footnote{We will see later, why it makes sense to also use the letter $C$ for a coparallel class.} If $C$ is a coparallel class of a grahpic matroid $\M(G)$, for a graph $G$, we will also say that $C$ is a coparallel class of $G$. Note that if an element $e \in E$ is neither a (co)loop, nor contained in a cocircuit of size $2$, then $\{e\}$ forms a coparallel class by itself. If $e \in E$ is a (co)loop, we do not consider $\{e\}$ to be a coparallel class and, hence, (co)loops are not contained in any coparallel class. Since coparallelism defines a transitive relation, the set $\CP(\M)$ of coparallel classes of $\M$ is a partition of the set $E$, as long as $\M$ does not contain any loops or coloops. If $\M=\M(G)$ is a graphic matroid, we will also write $\CP(G)$ instead of $\CP(\M(G))$. Then $\CP(G)$ is a partition of the edge set $E(G)$ if $G$ does not contain any bridges which is especially true if $G$ is $2$-edge-connected.\footnote{Recall that $G$ is simple and thus contains no selfloops.} In \cite{BARAHONA198640}, it has been shown that for a binary matroid $\M$ the dimension of $\Cyc(\M)$ is given by the number of coparallel classes $\vert \CP(\M) \vert$.

We want to point out some easy but helpful properties of coparallel classes. Let $C=\{e_1,\ldots,e_k\} \subseteq E(G)$ be a coparallel class of a $2$-edge-connected graph $G$ of size $k$, and see \Cref{structure of cpc} for visualizations. If $k=1$, $G-e_1$ is a $2$-edge-connected graph, since otherwise $e_1$ would be a bridge or contained in a $2$-cut. For $k=2$, we have that $G-\{e_1,e_2\}$ has two connected components $G_1$ and $G_2$, since $\{e_1,e_2\}$ is a minimal $2$-cut. Again, $G_1$ and $G_2$ are $2$-edge-connected, since otherwise $C$ would contain at least three elements. If $k=3$, again $G-\{e_1,e_2\}$ has two connected components. Now, since also $\{e_1,e_3\}$ and $\{e_2,e_3\}$ are $2$-cuts, $e_3$ is a bridge in one of the connected components of $G-\{e_1,e_2\}$. Therefore, $G$ has the structure seen in \Cref{k=3}, where, again, $G_1$, $G_2$ and $G_3$ are $2$-edge-connected for the same reason as before. Continuing this procedure, we see that for $k \ge 4$, the graph $G$ has the structure seen in \Cref{k=4}, i.e., $G-C$ has $k$ connected components $G_1,\ldots,G_k$, and if one contracts these connected components to single vertices, we obtain a cycle where the edges are exactly the edges of the coparallel class $C$. Furthermore, we have:

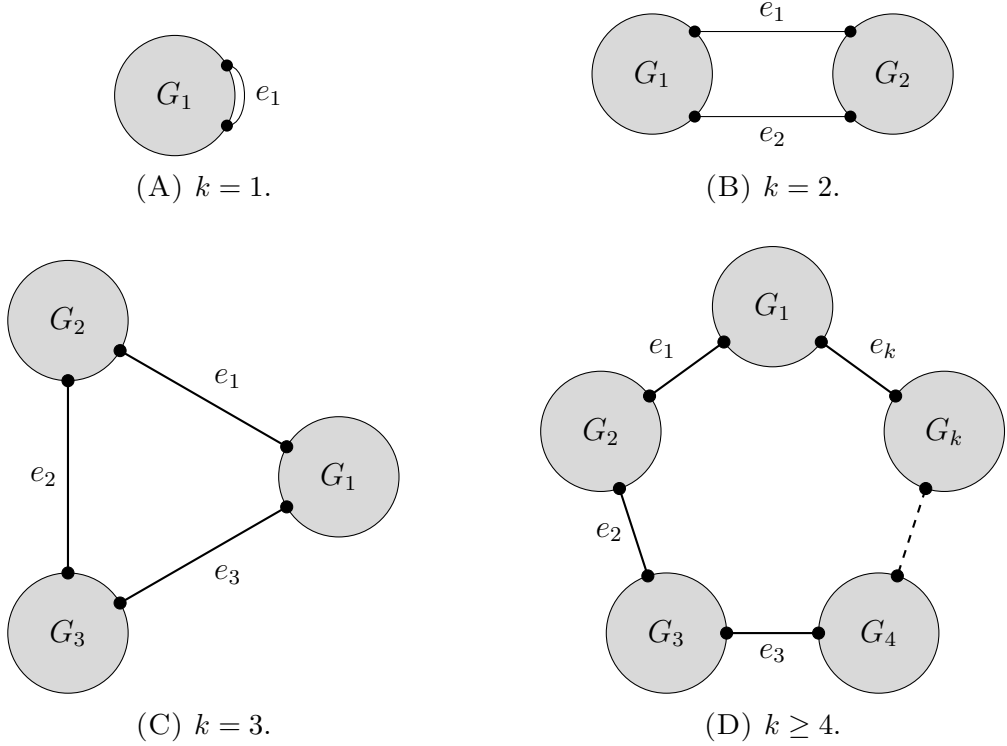
\begin{figure}
    \centering
    \begin{subfigure}[b]{0.45 \textwidth}
    \centering
        \begin{tikzpicture}[scale=1.6]
        \coordinate (A) at (0,0);
        \coordinate (A1) at ($(A)+(30:0.5)$);
        \coordinate (A2) at ($(A)+(330:0.5)$);
        
        \draw[fill=gray!30] (A) circle (0.5) node{$G_1$};
        \draw(A1) to [bend left=90] node[midway, right] {$e_1$} (A2);
        \fill (A1) circle (0.05);
        \fill (A2) circle (0.05);
        \end{tikzpicture}
        \caption{$k=1$.}
    \end{subfigure}
    \begin{subfigure}[b]{0.45 \textwidth}
    \centering
        \begin{tikzpicture}[scale=1.6]
        \coordinate (A) at (-1,0);
        \coordinate (B) at (1,0);

        \coordinate (A1) at ($(A)+(45:0.5)$);
        \coordinate (A2) at ($(A)+(360-45:0.5)$);
        \coordinate (B1) at ($(B)+(135:0.5)$);
        \coordinate (B2) at ($(B)+(180+45:0.5)$);
        
        \draw[fill=gray!30] (A) circle (0.5) node{$G_1$};
        \draw[fill=gray!30] (B) circle (0.5) node{$G_2$};

        \draw (A1) to node[midway, above] {$e_1$} (B1);
        \draw (A2) to node[midway, below] {$e_2$} (B2);
        \fill (A1) circle (0.05);
        \fill (A2) circle (0.05);
        \fill (B1) circle (0.05);
        \fill (B2) circle (0.05);
        \end{tikzpicture}
        \caption{$k=2$.}
    \end{subfigure}

    \hspace{1cm}

    \begin{subfigure}[b]{0.45 \textwidth}
    \centering
        \begin{tikzpicture}[scale=1.6]
        \coordinate (A) at (0:1.5);
        \coordinate (B) at (120:1.5);
        \coordinate (C) at (240:1.5);
        
        \draw[thick] (A) -- (B) -- (C) -- (A);

        \node[above right] at ($0.5*(A)+0.5*(B)$) {$e_1$};
        \node[left] at ($0.5*(B)+0.5*(C)$) {$e_2$};
        \node[below right] at ($0.5*(C)+0.5*(A)$) {$e_3$};

        \draw[fill=gray!30] (A) circle (0.5) node{$G_1$};
        \draw[fill=gray!30] (B) circle (0.5) node{$G_2$};
        \draw[fill=gray!30] (C) circle (0.5) node{$G_3$};
        
        \draw[fill=black] ($(A)+(150:0.5)$) circle (0.05);
        \draw[fill=black] ($(B)+(330:0.5)$) circle (0.05);
        \draw[fill=black] ($(B)+(270:0.5)$) circle (0.05);
        \draw[fill=black] ($(C)+(90:0.5)$) circle (0.05);
        \draw[fill=black] ($(C)+(30:0.5)$) circle (0.05);
        \draw[fill=black] ($(A)+(210:0.5)$) circle (0.05);
        
        \end{tikzpicture}
        \caption{$k=3$.}
        \label{k=3}
        \end{subfigure}
        \begin{subfigure}[b]{0.45 \textwidth}
        \centering
        \begin{tikzpicture}[scale=1.6]
        \coordinate (A) at (90:1.5);
        \coordinate (B) at (162:1.5);
        \coordinate (C) at (234:1.5);
        \coordinate (D) at (306:1.5);
        \coordinate (E) at (18:1.5);
        
        \draw[thick] (E) -- (A) -- (B) -- (C) -- (D);
        \draw[thick,dashed] (D) -- (E);

        \node[above left] at ($0.5*(A)+0.5*(B)$) {$e_1$};
        \node[left] at ($0.5*(B)+0.5*(C)$) {$e_2$};
        \node[below] at ($0.5*(C)+0.5*(D)$) {$e_3$};
        \node[above right] at ($0.5*(E)+0.5*(A)$) {$e_k$};

        \draw[fill=gray!30] (A) circle (0.5) node{$G_1$};
        \draw[fill=gray!30] (B) circle (0.5) node{$G_2$};
        \draw[fill=gray!30] (C) circle (0.5) node{$G_3$};
        \draw[fill=gray!30] (D) circle (0.5) node{$G_4$};
        \draw[fill=gray!30] (E) circle (0.5) node{$G_k$};
        
        \draw[fill=black] ($(A)+(216:0.5)$) circle (0.05);
        \draw[fill=black] ($(B)+(36:0.5)$) circle (0.05);
        \draw[fill=black] ($(B)+(360-72:0.5)$) circle (0.05);
        \draw[fill=black] ($(C)+(90+18:0.5)$) circle (0.05);
        \draw[fill=black] ($(C)+(0.5,0)$) circle (0.05);
        \draw[fill=black] ($(D)+(-0.5,0)$) circle (0.05);
        \draw[fill=black] ($(D)+(90-18:0.5)$) circle (0.05);
        \draw[fill=black] ($(E)+(270-18:0.5)$) circle (0.05);
        \draw[fill=black] ($(E)+(90+54:0.5)$) circle (0.05);
        \draw[fill=black] ($(A)+(360-36:0.5)$) circle (0.05);
        
        \end{tikzpicture}
        \caption{$k \ge 4$.}
        \label{k=4}
        \end{subfigure}
    \caption{The structure of a coparallel class $C=\{e_1,\ldots,e_k\}$ of a graph $G$. Every grey filled cycle represents a $2$-edge-connected subgraph of $G$.}
    \label{structure of cpc}
\end{figure}

\begin{observation}
\label{G-C}
Let $G$ be $2$-edge-connected and $C \in \CP(G)$ be a coparallel class. Then the connected components of $G-C$ are $2$-edge-connected.
\end{observation}

Note that this observation also includes the case that a connected component of $G-C$ consists of a single vertex and no edge.

Next, we relate coparallel classes to $2$-edge-connected subgraphs. Let $H \subseteq G$ be a $2$-edge-connected subgraph of a graph $G$ and let $C \in \CP(G)$ be a coparallel class of $G$. Assume, that $e_1,e_2 \in C$ are distinct edges of $C$, such that $e_1 \in E(H)$, but $e_2 \notin E(H)$. Since $\{e_1,e_2\}$ is a $2$-cut of $G$, that means that $e_1$ is now a bridge in $H$, which is a contradiction. The latter already yields the following.

\begin{observation}
\label{completelycontained}
Let $H \subseteq G$ be a $2$-edge-connected subgraph of a graph $G$. Then for every coparallel class $C \in \CP(G)$ of $G$ we either have $C \subseteq E(H)$ or $C \cap E(H) = \emptyset$. In particular, for a coparallel class $C \in \CP(G)$ of a $2$-edge-connected graph $G$, every other coparallel class $C' \in \CP(G)\setminus \{C\}$ is completely contained in a single connected component of $G-C$.
\end{observation}

By \Cref{completelycontained}, given a $2$-edge-connected subgraph $H \subseteq G$ of a graph $G$, we can talk about the coparallel classes of $G$ contained in $H$. We will denote this set by $\CP_H(G)$ and formally this set is given by
\begin{align*}
	\CP_H(G)\coloneqq \CP(G) \cap 2^{E(H)}.
\end{align*}
Note that in general the equality $\CP_H(G)=\CP(H)$ does \emph{not} hold. Given two edges that are coparallel in $G$ and contained in $H$, these edges are also coparallel in $H$. Furthermore, since all coparallel classes of $G$ are either completely contained in $H$ or not at all, this implies that the coparallel classes of $H$ are unions of coparallel classes of $G$. In particular, we have $\vert \CP(H) \vert \le \vert \CP_H(G) \vert$. We will especially discuss this situation in the case that $H$ is a connected component of $G-C$ for a coparallel class $C \in \CP(G)$, cf. \Cref{cpc of G-C fig}. Suppose there is a connected component $G_1$ of $G-C$ containing edges. By \Cref{G-C}, we know that $G_1$ is $2$-edge-connected. Let $C'=\{f_1,\ldots,f_m\} \in \CP(G_1)$ be a coparallel class of $G_1$. Then again we get the structure seen in \Cref{structure of cpc}, i.e., $G_1-C'$ consists of $2$-edge-connected components $G'_1,\ldots,G'_m$, where one obtains a cycle (for $m \ge 3$, two parallel edges for $m=2$, a loop for $m=1$, resp.) if one contracts $G'_1,\ldots,G'_m$ to single vertices. Assume w.l.o.g. that $e_1$ and $e_k$ are the two edges of $C$ touching $G_1$ and that $e_1$ is connected to $G'_1$. The edge $e_k$ can be connected to any component $G'_\ell$, $1\le \ell \le m$. In $G$, the set $C'$ is only a coparallel class if $\ell=1$. In all other cases, $C'$ is not a coparallel class in $G$ but the sets $\{f_1,\ldots,f_{\ell-1}\}$ and $\{f_\ell,\ldots,f_m\}$ are coparallel classes in $G$ (sticking to the notation of \Cref{cpc of G-C fig}). This already shows the following

\begin{observation}
\label{cpc of G-C obs}
Let $G$ be $2$-edge-connected, $C \in \CP(G)$ be a coparallel class of $G$, and $G'$ be a connected component of $G-C$ that contains edges. If a coparallel class $C'$ of $G'$ is not a coparallel class of $G$, then $C'=C_1 \dot{\cup} C_2$, where $C_1$ and $C_2$ are coparallel classes of $G$.
\end{observation}

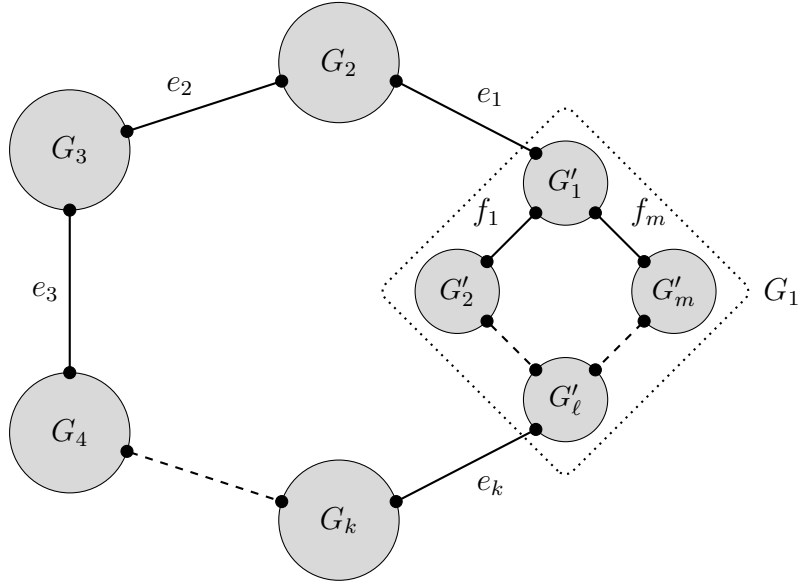
\begin{figure}
    \centering
        \begin{tikzpicture}[scale=1.6]
        \coordinate (A) at (0:2);
        \coordinate (B) at (72:2);
        \coordinate (C) at (144:2);
        \coordinate (D) at (216:2);
        \coordinate (E) at (288:2);
        
        \coordinate (S) at (2.5,0);
        \coordinate (1) at ($(S)+(90:0.9)$);
        \coordinate (2) at ($(S)+(180:0.9)$);
        \coordinate (3) at ($(S)+(270:0.9)$);
        \coordinate (4) at ($(S)+(0:0.9)$);
        
        \coordinate (V1) at ($(1)+(135:0.35)$);
        \coordinate (V2) at ($(3)+(225:0.35)$);

        \coordinate(G) at (-1.4265847745/2,1.9635254916/2);
        \coordinate(H) at (0,-1.2135254916);
        \coordinate(I) at (1.4265847745/2,1.9635254916/2);
        
        \draw[thick] (B) -- (C) -- (D);
        \draw[thick,dashed] (D) -- (E);
        
        \draw[thick] (V1) -- ($(B)+(-18:0.5)$);
        \draw[thick] (V2) -- ($(E)+(18:0.5)$);
        
        %\node[above left] at ($0.5*(A)+0.5*(B)$) {$e_1$};
        \node[above left] at ($0.5*(B)+0.5*(C)$) {$e_2$};
        \node[left] at ($0.5*(C)+0.5*(D)$) {$e_3$};
        %\node[above right] at ($0.5*(E)+0.5*(A)$) {$e_n$};

        % \draw[fill=gray!30] (A) circle (0.5) node{$K_1$};
        \draw[fill=gray!30] (B) circle (0.5) node{$G_2$};
        \draw[fill=gray!30] (C) circle (0.5) node{$G_3$};
        \draw[fill=gray!30] (D) circle (0.5) node{$G_4$};
        \draw[fill=gray!30] (E) circle (0.5) node{$G_k$};
        
        % \draw[fill=black] ($(A)+(216:0.5)$) circle (0.05);
        \draw[fill=black] ($(B)+(180+18:0.5)$) circle (0.05);
        %\draw[fill=black] ($(B)+(360-72:0.5)$) circle (0.05);
        \draw[fill=black] ($(C)+(18:0.5)$) circle (0.05);
        \draw[fill=black] ($(C)+(270:0.5)$) circle (0.05);
        \draw[fill=black] ($(D)+(90:0.5)$) circle (0.05);
        \draw[fill=black] ($(D)+(-18:0.5)$) circle (0.05);
        \draw[fill=black] ($(E)+(180-18:0.5)$) circle (0.05);
        %\draw[fill=black] ($(E)+(90+54:0.5)$) circle (0.05);
        % \draw[fill=black] ($(A)+(360-36:0.5)$) circle (0.05);
        
        \draw[fill=black] ($(B)+(-18:0.5)$) circle (0.05);
        \draw[fill=black] ($(E)+(18:0.5)$) circle (0.05);

        \draw[thick] (4) -- (1) -- (2);
        \draw[thick,dashed] (2) -- (3) -- (4);
        
        \node[above left] at ($0.5*(1)+0.5*(2)$) {$f_1$};
        \node[above right] at ($0.5*(1)+0.5*(4)$) {$f_m$};
        
        \draw[fill=gray!30] (1) circle (0.35) node{\small$G'_1$};
        \draw[fill=gray!30] (2) circle (0.35) node{\small$G'_2$};
        \draw[fill=gray!30] (3) circle (0.35) node{\small$G'_\ell$};
        \draw[fill=gray!30] (4) circle (0.35) node{\small$G'_m$};
        
        \draw[fill=black] (V1) circle (0.05);
        \draw[fill=black] (V2) circle (0.05);
        
        \draw[fill=black] ($(1)+(180+45:0.35)$) circle (0.05);
        \draw[fill=black] ($(1)+(-45:0.35)$) circle (0.05);
        \draw[fill=black] ($(2)+(45:0.35)$) circle (0.05);
        \draw[fill=black] ($(2)+(-45:0.35)$) circle (0.05);
        \draw[fill=black] ($(3)+(45:0.35)$) circle (0.05);
        \draw[fill=black] ($(3)+(180-45:0.35)$) circle (0.05);
        \draw[fill=black] ($(4)+(180-45:0.35)$) circle (0.05);
        \draw[fill=black] ($(4)+(180+45:0.35)$) circle (0.05);
        
        \node[above right] at ($0.5*(V1)+0.5*(B)+0.5*(-18:0.5)$) {$e_1$};
        \node[below right] at ($0.5*(V2)+0.5*(E)+0.5*(18:0.5)$) {$e_k$};

        \draw[dotted, thick, rounded corners, rotate around={45:(2.5,0)}] (1.4,-1.1) rectangle (3.6,1.1);

        \node at (4.3,0) {$G_1$};
        
        \end{tikzpicture}
    \caption{The structure of the coparallel classes of $G-C$ for $C=\{e_1,\ldots,e_k\} \in \CP(G)$. The coparallel class $C'=\{f_1,\ldots,f_\ell,\ldots,f_m\}$ is completely contained in $G_1$}
    \label{cpc of G-C fig}
\end{figure}

The following is a direct consequence of \Cref{completelycontained}.

\begin{observation}
\label{observation}
Let $G$ be a graph and $C \in \CP(G)$ be a coparallel class of~$G$. For all $x \in \TECSP(G)$, and all $e,f \in C$ we have $x_e=x_f$ .
\end{observation}

\Cref{observation} allows us to determine the dimension of $\TECSP(G)$.

\begin{theorem}
\label{dimension}
For a graph $G$, we have
\begin{align*}
	\dim(\mathrm{TECSP}(G))=\vert \CP(G) \vert .
\end{align*}
\end{theorem}

\begin{proof}
	Suppose that $e \in E(G)$ is a bridge of $G$. Then we have $x_e=0$ for all $x \in \TECSP(G)$, since bridges are never contained in $2$-edge-connected subgraphs. Moreover, bridges do not have any impact on the number of coparallel classes of a graph $G$. In fact, we can delete all bridges and assume that $G$ does not contain any bridges. Now, by \Cref{observation} and the fact that $\CP(G)$ forms a partition of $E(G)$, we have $\dim(\mathrm{TECSP}(G)) \le \vert \CP(G)\vert$. Assuming $G$ does not contain any bridges, all connected components $G_1,\ldots,G_m$ of $G$ are $2$-edge-connected. Let $C \in \CP(G)$ be a coparallel class of $G$. Then $C \subseteq E(G_i)$ for some $i \in [m]$ due to \Cref{completelycontained}. Let $H_1,\ldots,H_k$ denote the connected components of $G_i-C$. Due to \Cref{G-C}, $\chi^{H_j} \in \TECSP(G)$ for all $j \in [k]$. Since $\textbf{0} \in \TECSP(G)$, $\aff(\TECSP(G))$ is a linear subspace of $\R^{E(G)}$ and $\chi^C=\chi^{G_i}-\sum_{j=1}^k \chi^{H_j} \in \aff(\TECSP(G))$, which shows $\dim(\aff(\TECSP(G))) \ge \vert \CP(G) \vert$.
\end{proof}

\begin{remark}
We can also define the $2$-edge-connected subgraph polytope for a \emph{multigraph}~$G$, i.e., a graph with multi-edges between common vertex pairs. In $G$, we may replace every multi-edge by a path of length two. Such a path consists of two coparallel edges. Denoting the obtained simple graph by $G'$, the polytopes $\TECSP(G)$ and $\TECSP(G')$ are affinely isomorphic due to \Cref{observation}. Therefore, we also cover multigraphs in this article. To simplify notation, we will always assume that the considered graphs are simple in the following.
\end{remark}

Before we end the preliminaries, we state the following corrected version of \cite[Corollary 4.23]{BARAHONA198640}, which will be needed in one of the proofs. Let us denote by $\ESP(G)$ the convex hull of incidence vectors of (not necessarily connected) Eulerian subgraphs of a graph $G$, i.e., subgraphs in which every vertex has even degree. It is well known that every Eulerian subgraph is an edge-disjoint union of cycles. For an explanation on the correction, see the appendix (\Cref{appendix}).

\begin{theorem}[\cite{BARAHONA198640}]
    \label{facetsESP}
    Let $G$ be a graph, let $E_3 \subseteq E(G)$ be the edges contained in cuts of size $3$ and let $E' \subseteq E(G)$ be an inclusion-wise maximal set of edges not containing bridges or $2$-cuts. Then $\ESP(G)$ is given by
    \begin{enumerate}[(a)]
        \item $x_e=0$ for each bridge $e \in E(G)$,
        \item $x_{e_1}=x_{e_i}$ for each coparallel class $C=\{e_1,\ldots,e_k\} \in \CP(G)$ and all $i=2,\ldots,k$,
        \item $0 \le x_e \le 1$ for each $e \in E(G) \setminus E_3$ that is not a bridge,
        \item $x(F)-x(C\setminus F) \le \vert F \vert -1$ for each minimal cut $C \subseteq E'$, and with property (P), and each $F \subseteq C$, $\vert F \vert$ odd; Hereby property (P) is that if $C$ has a chord defined by cuts $F_1,F_2 \subseteq E(G)$, then $F_1$ or $F_2$ has cardinality $2$.
    \end{enumerate}
    Moreover, the system above is nonredundant.
\end{theorem}

Since this is a nonredundant system, the inequalities \textit{(c)} and \textit{(d)} are actually facet-defining for $\ESP(G)$.

\section{Facets}

\label{facets}

We study different types of supporting hyperplanes for the $2$-edge-connected subgraph polytope $\TECSP(G)$ and when they define facets. In the following we will always assume that the underlying graph $G$ is $2$-edge-connected as already justified in the introduction.

\subsection{The box inequalities}

We characterize when the box inequalities $0 \le x_e \le 1$ define facets of $\TECSP(G)$.

\begin{theorem}
    \label{boxineq0}
    Let $G$ be a $2$-edge-connected graph, $C \in \CP(G)$ and $e\in C$. Then the inequality
    \begin{align*}
        x_e \ge 0
    \end{align*}
    defines a facet of $\TECSP(G)$ if and only if $e$ is not contained in a $3$-cut.
\end{theorem}

\begin{proof}
    Suppose that $e$ is contained in a $3$-cut $\delta(S)=\{e,e_1,e_2\}$ for some $ \emptyset \subsetneq S \subsetneq V(G)$. We consider the inequalities
    \begin{align*}
        x_{e_1}-x_{e_2}-x_e \le 0
    \end{align*}
    and
    \begin{align*}
        x_{e_2}-x_{e_1}-x_e \le 0.
    \end{align*}
    These are both supporting hyperplanes of $\TECSP(G)$ since $2$-edge-connected subgraphs $H$ of $G$ cannot contain exactly one edge of a cut, as otherwise $H$ would contain a bridge. In fact, they are special cases of the so-called \emph{asymmetric cut inequalities} that will be studied in \Cref{section asymmetric cut}. These two inequalities define different proper faces of $\TECSP(G)$, since e.g. $\chi^G$ is not contained in the corresponding face, and they already imply $x_e \ge 0$. Hence, the latter does not define a facet.

    Suppose now that $e$ is not contained in a $3$-cut and let $\F \subseteq \TECSP(G)$ denote the face defined by $x_e=0$. By an analogous argument as in \cite{BARAHONA198640}, we get the result as follows. Let $C \in \CP(G)$ be the coparallel class that contains $e$. $\F$ is affinely isomorphic to the $2$-edge-connected subgraph polytope $\TECSP(G-C)$ of the graph $G-C$. Since $e$ is not contained in a $3$-cut, the same holds for any $e' \in C$, and therefore, two edges $f_1,f_2 \in E(G) \setminus C$ are coparallel in $G$ if and only if they are coparallel in $G-C$. This implies $\vert \CP(G-C) \vert =\vert \CP(G) \vert-1$ and the claim follows by \Cref{dimension}.
        
    % Suppose now that $e$ is not contained in any $3$-cut. We consider the cycle polytope $\Cyc(\M(G))$ of the graphic matroid $\M(G)$. Clearly, the inequality $x_e \ge 0$ also defines a supporting hyperplane for $\Cyc(\M(G))$. Let $\F_{\Cyc(\M(G))}$ and $\F_{\TECSP(G)}$ denote the corresponding faces of $\Cyc(\M(G))$ and $\TECSP(G)$, respectively, given by $x_e = 0$. In \cite[Theorem 4.8]{BARAHONA198640}, Barahona and Gr\"otschel showed that the inequality $x_e \ge 0$ defines a facet of $\Cyc(\M(G))$ if $e$ is not contained in a $3$-cut. We will show that $\aff(\F_{\Cyc(\M(G))}) \subseteq \aff(\F_{\TECSP(G)})$ which implies that since $\F_{\Cyc(\M(G))}$ is a facet, so is $\F_{\TECSP(G)}$, since both polytopes have the same dimension. For this aim let $H \subseteq G$ be a subgraph of $G$ such that $\chi^H$ is a vertex of $\F_{\Cyc(\M(G))}$. Then $H$ is the edge-disjoint union of cycles $C_1,\ldots,C_k \subseteq G$ and none of these cycles containes the edge $e$. This means that $\chi^{C_i}$ is a vertex of $\F_{\TECSP(G)}$ for all $i \in [k]$. Since the all-zeros vector $\mathbf{0}$ is contained in $\F_{\TECSP(G)}$, $\aff(F_{\TECSP(G)})$ is a linear subspace and $\chi^H=\sum_{i=1}^k \chi^{C_i} \in \aff(\F_{\TECSP(G)})$.
\end{proof}

%\textcolor{red}{Muss man das genauer mit symmetrischer Differenz begründen?}

\begin{theorem}
    \label{boxineq1}
    Let $G$ be a $2$-edge-connected graph and $e \in E(G)$. Then the inequality
    \begin{align*}
        x_e \le 1
    \end{align*}
    defines a facet of $\TECSP(G)$.
\end{theorem}

\begin{proof}
    Let $\F \subseteq \TECSP(G)$ denote the face defined by $x_e = 1$. Let $\CP(G)=\{C_e,C_1,\ldots,C_\ell\}$, where $C_e$ is the coparallel class containing $e$. For $j \in [\ell]$, we denote by $G_j^e$ the connected component of the graph $G-C_j$ containing $e$. Since $G_j^e$ is $2$-edge-connected (by \Cref{G-C}), $\chi^G$ and $\chi^{G_j^e}$, for $j=1,\ldots,\ell$, are vertices of $\F$. We now show that these $\vert \CP(G)\vert$ many vectors are affinely independent, which by \Cref{dimension} implies that $\F$ is a facet. For this aim, we define a partial order on the coparallel classes $C_1,\ldots,C_\ell$ as follows: for $i,j \in [\ell],i\neq j$, let
    \begin{align*}
        C_i <_e C_j \Longleftrightarrow C_i \nsubseteq E(G_j^e).
    \end{align*}
    To show that $<_e$ defines a partial order, we show the following claim.
    \begin{align*}
        C_i <_e C_j \Longrightarrow G_j^e \subseteq G_i^e \text{ and } C_j \subseteq E(G_i^e),
    \end{align*}
    which can be seen as follows. Suppose $C_i <_e C_j$. Then by definition $C_i \nsubseteq E(G_j^e)$. Due to \Cref{completelycontained}, this already means that $C_i \cap E(G_j^e)=\emptyset$. Hence, $G-C_i$ contains all edges of $G_j^e$ which implies that $G_j^e \subseteq G_i^e$. Since at least one edge $f$ of $C_j$ is incident to an edge of $G_j^e$ and since $f$ is contained in $G-C_i$ it follows that $f \in E(G_i^e)$. By \Cref{completelycontained}, we conclude that $C_j \subseteq E(G_i^e)$.
    
    Using this claim we can easily verify that $<_e$ defines a partial order on $C_1,\ldots,C_\ell$: If $C_i <_e C_j$, then $C_j \subseteq E(G_i^e)$ and hence, $C_j \nless_e C_i$. This shows that $<_e$ is antisymmetric. Suppose now that $C_i <_e C_j$ and $C_j <_e C_k$. From the first relation we have $C_i \nsubseteq E(G_j^e)$ by definition and from the second relation we have $G_k^e \subseteq G_j^e$ by the claim above. Hence, $C_i \nsubseteq E(G_k^e)$ which implies $C_i <_e C_k$ and in particular, $<_e$ is transitive.
    
    The vectors $\chi^{G_j^e}$ satisfy
    \begin{align*}
        \chi_f^{G_j^e}=\begin{cases}
            1, & \text{ for all } f \in C_e, \\
            0, & \text{ for all } C_i \le_e C_j \text{ and all } f \in C_i, \text{ and} \\
            1, & \text{ for all } C_i \nleq_e C_j \text{ and all } f \in C_i,
        \end{cases}
    \end{align*}
    for all $j \in [\ell]$.
    
    Let $\pi \colon \TECSP(G) \to \R^{\ell}$ be the map defined by $\pi(x)_i \coloneqq x_f$, for any $f \in C_i$, for all $i \in [\ell]$. By \Cref{observation}, this is well-defined. Furthermore, we have
    \begin{align*}
        \dim(\aff(\chi^G,\chi^{G_1^e},\ldots,\chi^{G_\ell^e}))&=\dim(\aff(\pi(\chi^G),\pi(\chi^{G_1^e}),\ldots,\pi(\chi^{G_\ell^e})))\\ &=\dim(\mathrm{span}(\pi(\chi^G)-\pi(\chi^{G_1^e}),\ldots,\pi(\chi^G)-\pi(\chi^{G_\ell^e}))).
    \end{align*}
    Without loss of generality we can assume that our indices satisfy $i \le j$ if $C_i \le C_j$. Writing the vectors $\pi(\chi^G)-\pi(\chi^{G_1^e}),\ldots,\pi(\chi^G)-\pi(\chi^{G_\ell^e})$ row-wise in a matrix, yields a lower triangular $\ell \times \ell$-matrix that has ones in the diagonal. Therefore, the matrix has rank $\ell$, which finishes the proof.
\end{proof}

\subsection{The asymmetric cut inequalities}

\label{section asymmetric cut}

In the following we will focus on inequalities that are defined by cuts. Let $H \subseteq G$ both be $2$-edge-connected and let $\delta(S)$ be a cut for some $\emptyset \subsetneq S \subsetneq V(G)$. It is not possible that $E(H)$ contains only one edge $e \in \delta(S)$ because otherwise $H$ has a bridge and is not $2$-edge-connected. Therefore, for the vertex $\chi^H$ of $\TECSP(G)$ the inequality
\begin{align*}
    \chi_e^H-\chi^H(\delta(S)\setminus \{e\}) \le 0
\end{align*}
holds. This shows that for $\emptyset \subsetneq S \subsetneq V(G)$ and $e \in \delta(S)$ the inequality
\begin{align}
    \label{asymmetric}
    x_e-x(\delta(S)\setminus\{e\}) \le 0
\end{align}
defines a supporting hyperplane of $\TECSP(G)$. An inequality as in \eqref{asymmetric} will be called an \textit{asymmetric cut inequality}. In the following we study, when this inequality defines a facet. We can assume that $\vert\delta(S)\vert \ge 3$: $\vert \delta(S)\vert=1$ is not possible since we only consider $2$-edge-connected graphs; suppose that $\vert \delta(S) \vert =2$. Then every $2$-edge-connected subgraph of $G$ contains either both edges of $\delta(S)$ or none of them. Hence, every vertex of $\TECSP(G)$ is contained in the face defined by the corresponding asymmetric cut inequality. This already shows

\begin{observation}
\label{cond1}
Let $G$ be $2$-edge-connected and $\emptyset \subsetneq S \subsetneq V$ such that $\vert\delta(S)\vert=2$. Then the face~$\F \subseteq \TECSP(G)$ defined by the asymmetric cut inequality
\begin{align*}
    x_e-x(\delta(S)\setminus\{e\}) \le 0
\end{align*}
is given by $\F=\TECSP(G)$. In particular, the asymmetric cut inequality does not define a facet of $\TECSP(G)$ in this case.
\end{observation}

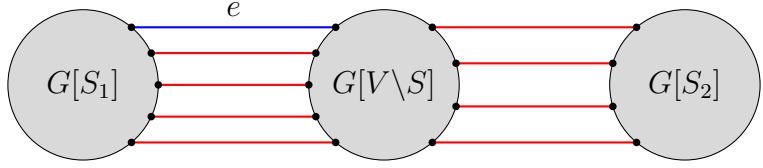
\begin{figure}
    \centering
        \begin{tikzpicture}[scale=1]
            \draw[fill=gray!30] (-2, 0) circle (1) node {$G[S_1]$};
            \draw[fill=gray!30] (2, 0) circle (1) node {$G[V {\setminus} S]$};
            \draw[fill=gray!30] (6, 0) circle (1) node {$G[S_2]$};

            \coordinate (v1) at ($(-2,0)+(50:1)$);
            \coordinate (v2) at ($(-2,0)+(25:1)$);
            \coordinate (v3) at ($(-2,0)+(0:1)$);
            \coordinate (v4) at ($(-2,0)+(-25:1)$);
            \coordinate (v5) at ($(-2,0)+(-50:1)$);

            \coordinate (w1) at ($(2,0)+(130:1)$);
            \coordinate (w2) at ($(2,0)+(155:1)$);
            \coordinate (w3) at ($(2,0)+(180:1)$);
            \coordinate (w4) at ($(2,0)+(205:1)$);
            \coordinate (w5) at ($(2,0)+(230:1)$);
            
            \draw[blue, thick] (v1) -- (w1) node[midway, above, black] {$e$};
            \fill (v1) circle (0.05);
            \fill (w1) circle (0.05);
            
            \draw[red, thick] (v2) -- (w2);
            \fill (v2) circle (0.05);
            \fill (w2) circle (0.05);
            
            \draw[red, thick] (v3) -- (w3);
            \fill (v3) circle (0.05);
            \fill (w3) circle (0.05);

            \draw[red, thick] (v4) -- (w4);
            \fill (v4) circle (0.05);
            \fill (w4) circle (0.05);

            \draw[red, thick] (v5) -- (w5);
            \fill (v5) circle (0.05);
            \fill (w5) circle (0.05);

            \coordinate (x1) at ($(2,0)+(50:1)$);
            \coordinate (x2) at ($(2,0)+(16.66:1)$);
            \coordinate (x3) at ($(2,0)+(-16.66:1)$);
            \coordinate (x4) at ($(2,0)+(-50:1)$);

            \coordinate (y1) at ($(6,0)+(130:1)$);
            \coordinate (y2) at ($(6,0)+(180-16.66:1)$);
            \coordinate (y3) at ($(6,0)+(180+16.66:1)$);
            \coordinate (y4) at ($(6,0)+(230:1)$);

            \draw[red, thick] (x1) -- (y1);
            \fill (x1) circle (0.05);
            \fill (y1) circle (0.05);
            
            \draw[red, thick] (x2) -- (y2);
            \fill (x2) circle (0.05);
            \fill (y2) circle (0.05);
            
            \draw[red, thick] (x3) -- (y3);
            \fill (x3) circle (0.05);
            \fill (y3) circle (0.05);

            \draw[red, thick] (x4) -- (y4);
            \fill (x4) circle (0.05);
            \fill (y4) circle (0.05);

        \end{tikzpicture}
    \caption{Example for an asymmetric cut inequality with a non-minimal cut $\delta(S) \supsetneq \delta(S_1)$, where $S=S_1 \dot{\cup} S_2$.}
    \label{Disconnected induced subgraph}
\end{figure}

Furthermore we can assume that the set $\delta(S)$ is a minimal cut set in the sense that there is no set $F \subsetneq \delta(S)$ that is also a cut for the following reason. Suppose such an $F$ exists. We can assume w.l.o.g. that $e \in F$ since if $F$ is a cut, so is $\delta(S) \triangle F = \delta(S) \setminus F$ and one of these two sets contains $e$. But then the inequality
\begin{align}
    \label{dom}
    x_e-x(F \setminus \{e\}) \le 0
\end{align}
is also an asymmetric cut inequality and every incidence vector $\chi^H$ of a $2$-edge-connected subgraph $H \subseteq G$ that belongs to the face defined by \eqref{asymmetric} also belongs to the face defined by \eqref{dom}. This means that inequality \eqref{dom} dominates inequality \eqref{asymmetric} and we can study the latter instead. An example for such a behavior can be seen in \Cref{Disconnected induced subgraph}, where $S=S_1 \dot{\cup} S_2$, $\delta(S_1) \subsetneq \delta(S)$.

From now on, we hence assume that $\vert \delta(S) \vert \ge 3$ and $\delta(S)$ is minimal, and the latter is equivalent to $G[S]$ and $G[V(G) \setminus S]$ being connected. The next theorem characterizes when asymmetric cut inequalities define facets in this setting.

\begin{theorem}
    \label{facet asymmetric cut}
    Let $G$ be $2$-edge-connected and $\emptyset \subsetneq S \subsetneq V(G)$ such that $\delta(S)$ is a minimal cut set with $\vert \delta(S) \vert \ge 3$. For $e \in \delta(S)$ the asymmetric cut inequality
    \begin{align*}
        x_e-x(\delta(S)\backslash e) \le 0
    \end{align*}
    defines a facet of $\TECSP(G)$ if and only if every bridge in $G[S]$ and $G[V(G) \backslash S]$ lies in a coparallel class together with an edge of $\delta(S)$.
\end{theorem}

Note that the condition in the theorem above is equivalent to the property that if $\delta(S)$ has a chord, i.e., there exist cuts $F_1,F_2 \subseteq E(G)$, such that $F_1 \triangle F_2=\delta(S)$ and $F_1 \cap F_2=\{h\}$ for an $h \in E(G)$, then $F_1$ or $F_2$ has cardinality $2$, see below.

\begin{proof}
Suppose first that $G[S]$ contains a bridge $f \in E(G[S])$ that is not contained in a coparallel class with any edge of $\delta(S)$. Since $\delta(S)$ is a minimal cut set, we know that $G[S]$ and $G[V(G) \setminus S]$ are connected. Then the graph $G[S]-f$ has two connected components $G_1,G_2$. For $i=1,2$, let $E_i$ be the edges of $\delta(S)$ incident to $G_i$. Since $f$ is not in a $2$-cut with any edge of $\delta(S)$, we have $\vert E_1\vert,\vert E_2\vert \ge 2$. Moreover, we have $\delta(V(G_1))=E_1 \cup \{f\}$ and $\delta(V(G_2))=E_2 \cup \{f\}$ and without loss of generality we can assume that $e \in E_1$. An example for the described situation is shown in \Cref{bridge}. Then we have that
    \begin{align*}
        x_e-x(\delta(S)\setminus\{e\}) \le 0
    \intertext{is the sum of the two asymmetric cut inequalities}
        x_e-x(\delta(V(G_1))\setminus\{e\}) \le 0,\\ x_f-x(\delta(V(G_2))\setminus\{f\}) \le 0.
    \end{align*}
    Since $\vert E_1\vert,\vert E_2\vert \ge 2$, both of the two latter asymmetric cut inequalities define proper faces of $\TECSP(G)$ and it is clear that these two faces are different. This implies that the first asymmetric cut inequality is not a facet.

    We now show that in all other cases, the asymmetric cut inequality defines a facet of $\TECSP(G)$. The asymmetric cut inequality is actually a special case of inequality \textit{(d)} in \Cref{facetsESP} and therefore defines a facet of $\ESP(G)$ due to the remark directly below \Cref{facet asymmetric cut}. Let $\F_{\ESP(G)}$ and $\F_{\TECSP(G)}$ denote the face defined by the asymmetric cut inequality of $\ESP(G)$ and of $\TECSP(G)$, respectively. We will show that $\aff(\F_{\ESP(G)}) \subseteq \aff(\F_{\TECSP(G)})$, which implies that $\F_{\TECSP(G)}$ is a facet of $\TECSP(G)$, since $\F_{\ESP(G)}$ is a facet of $\ESP(G)$, and both polytopes have the same dimension due to \cite[Theorem 4.1]{BARAHONA198640} and \Cref{dimension}. Let $H \subseteq G$ be a subgraph of $G$ such that $\chi^H$ is a vertex of $\F_{\ESP(G)}$. This implies that $H$ either contains no edges of $\delta(S)$ or $H$ contains the edge $e \in \delta(S)$ and exactly one other edge $f \in \delta(S)\setminus \{e\}$. Since the connected components of $H$ are Eulerian, $H$ is the edge-disjoint union of cycles $C_1,\ldots,C_k \subseteq G$. If $E(H) \cap \delta(S)=\emptyset$, it is clear that $\chi^{C_i} \in \F_{\TECSP(G)}$ for all $i \in [k]$. If $e \in E(H)$, there is exactly one cycle $C_j$ that contains $e$ for some $j \in [k]$. Then $C_j$ also contains the edge $f$, because $E(C_j) \cap \delta(S)$ has even cardinality, and again we have $\chi^{C_i} \in \F_{\TECSP(G)}$ for all $i \in [k]$, since all other cycles do not contain any edge of $\delta(S)$. Since the all-zeros vector $\mathbf{0}$ is contained in $\F_{\TECSP(G)}$, $\aff(\F_{\TECSP(G)})$ is a linear subspace and also $\chi^H=\sum_{i=1}^k \chi^{C_i} \in \aff(\F_{\TECSP(G)})$.

\end{proof}

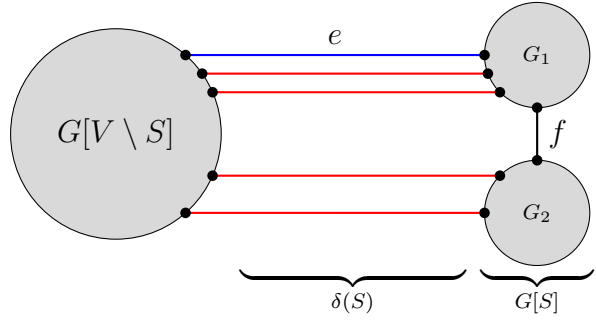
\begin{figure}
    \centering
    \begin{tikzpicture}[scale=1.4]
        \draw[fill=gray!30] (-2, 0) circle (1) node {$G[V\setminus S]$};
        \draw[fill=gray!30] (2, 0.75) circle (0.5) node {\tiny$G_1$};
        \draw[fill=gray!30] (2, -0.75) circle (0.5) node {\tiny$G_2$};

        \coordinate (v1) at ($(-2,0)+(48.59:1)$);
        \coordinate (v2) at ($(-2,0)+(34.98:1)$);
        \coordinate (v3) at ($(-2,0)+(23.36:1)$);
        \coordinate (v4) at ($(-2,0)+(-23.36:1)$);
        \coordinate (v5) at ($(-2,0)+(-48.59:1)$);

        \coordinate (w1) at ($(2,0.75)+(180:0.5)$);
        \coordinate (w2) at ($(2,0.75)+(200.7:0.5)$);
        \coordinate (w3) at ($(2,0.75)+(225:0.5)$);
        \coordinate (w4) at ($(2,-0.75)+(135:0.5)$);
        \coordinate (w5) at ($(2,-0.75)+(180:0.5)$);
        
        \draw[blue, thick] (v1) -- (w1) node[midway, above, black] {$e$};
        \fill (v1) circle (0.05);
        \fill (w1) circle (0.05);

        \foreach \i in {2,...,5}{
            \draw[red, thick] (v\i)--(w\i);
            \fill (v\i) circle (0.05);
            \fill (w\i) circle (0.05);
        }

        \draw[thick] (2,0.25) -- (2,-0.25) node[midway, right, black] {$f$};
        \fill (2,0.25) circle (0.05);
        \fill (2,-0.25) circle (0.05);

        \node at (2,-1.5) {$\underbrace{\hspace{1.5cm}}_{G[S]}$};
        \node[] at (0.25,-1.5) {$\underbrace{\hspace{3cm}}_{\delta(S)}$};
    	%\node at (-2,-1.5) {$\underbrace{\hspace{2.5cm}}_{G[V\setminus S]}$};
        \end{tikzpicture}
    \caption{Example for $G[S]$ containing a bridge $f$ that is not contained in a coparallel class with an edge of $\delta(S)$.}
    \label{bridge}
\end{figure}

\subsection{The connectivity cut inequalities}

We will now consider what we call the \emph{connectivity cut inequalities} and provide some motivation for these. So far, the asymmetric cut inequalities ensure that for a $2$-edge-connected graph $G$, no incidence vector of a subgraph $H \subseteq G$ that contains a bridge is feasible. However, a subgraph $H \subseteq G$ might not be $2$-edge-connected simply because it is not connected. As an example, consider the graph $G$ in \Cref{Example CCI}. The point $y \in \R^{E(G)}$ with $y_{f_1}=y_{f_2}=0$ and $y_e=1$ for all $e \in E(G)\setminus \{f_1,f_2\}$ satisfies all asymmetric cut inequalities since the subgraph $H \subseteq G$ with vertex set $V(H)=V(G)$ and edge set $E(H)=\mathrm{supp}(y)$ does not contain a bridge. However, $y$ does not lie in $\TECSP(G)$: Let $S \subseteq V(G)$ be such that $\delta(S)=\{f_1,f_2\}$. If a $2$-edge-connected subgraph $H \subseteq G$ contains an edge $e_1 \in E(G[S])$ and an edge $e_2 \in E(G[V(G)\setminus S])$, then $H$ has to contain at least two edges from the cut $\delta(S)$ in order to be $2$-edge-connected. This implies that the inequality
\begin{align*}
    2x_{e_1}-x_{f_1}-x_{f_2}+2x_{e_2} \le 2
\end{align*}
defines a supporting hyperplane for $\TECSP(G)$ and cuts off the otherwise feasible point~$y$. More generally, given a $2$-edge-connected graph $G$, and a set $\emptyset \subsetneq S \subsetneq V(G)$ with ${E(G[S])\neq \emptyset}$, and $E(G[V(G) \setminus S])\neq \emptyset$, $e_1 \in E(G[S])$ and $e_2 \in E(G[V(G) \setminus S])$, the inequality
\begin{align}
	\label{eq:CCI}
    2x_{e_1}-x(\delta(S))+2x_{e_2} \le 2
\end{align}
defines a supporting hyperplane of $\TECSP(G)$ by the arguments above. An inequality as in \eqref{eq:CCI} will be called a \emph{connectivity cut inequality}, since these inequalities ensure that only incidence vectors of connected subgraphs are feasible as we will see in \Cref{lattice points}. Before that, we study when such inequalities define facets.

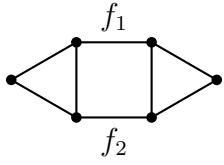
\begin{figure}
    \centering
    \begin{tikzpicture}
        \draw[thick] (0,0) -- (1,0);
        \draw[thick] (0,1) -- (1,1);
        \draw[thick] (0,1) -- (0,0);
        \draw[thick] (1,1) -- (1,0);
        \draw[thick] (0,0) -- (-0.8660254,0.5);
        \draw[thick] (0,1) -- (-0.8660254,0.5);
        \draw[thick] (1,0) -- (1.8660254,0.5);
        \draw[thick] (1,1) -- (1.8660254,0.5);
        
        \fill (0,0) circle (2pt);
        \fill (0,1) circle (2pt);
        \fill (-0.8660254,0.5) circle (2pt);
        \fill (1,1) circle (2pt);
        \fill (1,0) circle (2pt);
        \fill (1.8660254,0.5) circle (2pt);
        \coordinate[label=above:$f_1$] (A) at (0.5,1);
        \coordinate[label=below:$f_2$] (B) at (0.5,0);
    \end{tikzpicture}
    \caption{The graph used in the initial motivating example for connectivity cut inequalitites.}
    \label{Example CCI}
\end{figure}

Without loss of generality, we assume that the cut $\delta(S)$ is an inclusion-wise minimal cut. We will show that otherwise there exists a $\emptyset \subsetneq S' \subsetneq V(G)$ such that $\delta(S') \subsetneq \delta(S)$, and $e_1 \in G[S']$ and $e_2 \in G[V(G) \setminus S']$. If such $S'$ exists, the connectivity cut inequality
\begin{align*}
    2x_{e_1}-\delta(S')+2x_{e_2} \le 2
\end{align*}
dominates the original connectivity cut inequality \eqref{eq:CCI}, and we can study the latter instead. Suppose $\delta(S)$ is not inclusion-wise minimal, which means that the graph $G-\delta(S)$ has at least three connected components $G_1,G_2$, and $G_3$. After possibly renaming, we may assume $e_1 \in E(G_1)$, and $e_2 \in E(G_2)$. We claim that $\delta(V(G_1)) \subsetneq \delta(S)$ or $\delta(V(G_2)) \subsetneq \delta(S)$. Assume the contrary, namely $\delta(V(G_1))=\delta(S)=\delta(V(G_2))$. But this implies that each edge $f \in \delta(S)$ contains exactly one vertex of $V(G_1)$ and one vertex of $V(G_2)$, which means that there are no edges in $G$ that contain exactly one vertex of $V(G_3)$. This is not possible, since we only consider connected graphs $G$. Hence, we can choose $S'=V(G_1)$ or $S'=V(G)\setminus V(G_2)$, depending on which of $\delta(V(G_1))$ and $\delta(V(G_2))=\delta(V(G)\setminus V(G_2))$ is a proper subset of $\delta(S)$ (it may happen that both choices are valid). By construction, it is also clear that $e_1 \in G[S']$ and $e_2 \in G[V(G) \setminus S']$, and hence we find a dominating connectivity cut inequality, whenever $\delta(S)$ is not an inclusion-wise minimal cut.

We need to introduce further notation. With the notation above and assuming $\delta(S)$ is inclusion-wise minimal, for $W \in \{S,V(G) \setminus S\}$ and a bridge $f \notin \{e_1,e_2\}$ of $G[W]$, let $\delta(W,f) \subseteq \delta(S)$ denote the set of edges of $\delta(S)$ that have one end in the connected component of $G[W]-f$ containing neither $e_1$ nor $e_2$. With this we can provide the desired characterization.

\begin{theorem}
	\label{CCI Thm}
    Let $G$ be a $2$-edge-connected graph with vertex set $V$, let $\emptyset \subsetneq S \subsetneq V$, such that $\delta(S)$ is an inclusion-wise minimal cut with $e_1 \in E(G[S])$ and $e_2 \in E(G[V\backslash S])$. Then the connectivity cut inequality
    \begin{align}
        2x_{e_1}-x(\delta(S))+2x_{e_2} \le 2
        \label{CCI}
    \end{align}
    defines a facet of $\TECSP(G)$ if and only if one of the following two conditions holds:
    \begin{enumerate}[(i)]
    	\item $e_1$ or $e_2$ is a bridge in $G[S]$ or $G[V \setminus S]$, and $\vert \delta(S) \vert=2$;
    	\item neither $e_1$ nor $e_2$ is a bridge in $G[S]$ resp. $G[V \setminus S]$, and if there is a bridge $f$ in $G[W]$ for $W \in \{S,V \setminus S\}$, then $\vert \delta(W,f)\vert=1$.
    \end{enumerate}
\end{theorem}

We remark that in case \textit{(i)}, the considered connectivity cut inequality is equivalent to the box inequality $x_{e_2} \le 1$ or $x_{e_1} \le 1$ depending on whether $e_1$ is a bridge in $G[S]$ or $e_2$ is a bridge in $G[V \setminus S]$, respectively. Furthermore, if $G[S]$ and $G[V \setminus S]$ are both $2$-edge-connected, then there are no bridges and the considered connectivity cut always defines a facet.

\begin{proof}
    Let $\F \subseteq \TECSP(G)$ be the face of $\TECSP(G)$ defined by the considered connectivity cut inequality. We provide the proof in separate claims.
    
    \medskip
    
    \noindent \textbf{Claim 1}: If $e_1$ is a bridge in $G[S]$ or $e_2$ is a bridge in $G[V \setminus S]$ and $\vert \delta(S) \vert =2$, the considered connectivity cut inequality defines a facet of $\TECSP(G)$.
    
    \medskip
    
    \noindent \textit{Proof.} Without loss of generality we may assume that $e_1$ is a bridge in $G[S]$, since the condition is symmetric for $e_1$ and $e_2$. If $\vert \delta(S) \vert = 2$, the two edges of $\delta(S)$ and $e_1$ belong to the same coparallel class. The connectivity cut inequality thus simplifies to $x_{e_2} \le 1$, and therefore defines a facet by \Cref{boxineq1}. \hfill $\blacksquare$
    
    \medskip
    
    \noindent \textbf{Claim 2}: If $e_1$ is a bridge in $G[S]$ or $e_2$ is a bridge in $G[V \setminus S]$ and $\vert \delta(S) \vert \ge 3$, the considered connectivity cut inequality does not define a facet of $\TECSP(G)$.
    
    \medskip
    
    \noindent \textit{Proof.} Without loss of generality we may again assume that $e_1$ is a bridge in $G[S]$. Let $H \subseteq G$ be a $2$-edge-connected subgraph of $G$ such that $\chi^H \in \F$. This is only possible if $H$ contains at least one of $e_1$ and $e_2$. If $e_2 \in E(H)$, it is clear that $\chi^H \in \F_{e_2}$, where $\F_{e_2} \subseteq \TECSP(G)$ denotes the facet defined by $x_{e_2} \le 1$. Suppose that $H$ contains the edge $e_1$. In $G[S]$ the edge $e_1$ is a bridge, but, since $H$ is $2$-edge-connected, the edge $e_1$ has to be contained in a cycle in~$H$. Therefore, $E(H)$ intersects $\delta(S)$ non-trivially. Since $\chi^H \in \F$, $H$ also contains the edge $e_2$ and we have $\F \subseteq \F_{e_2}$. If $\vert \delta(S) \vert \ge 3$, $\chi^G \in \F_{e_2}$ but $\chi^G \notin \F$ and hence, $\F \subsetneq \F_{e_2}$ which shows that $\F$ is not a facet, since $\F_{e_2}$ is a proper face of $\TECSP(G)$ by \Cref{boxineq1}. \hfill $\blacksquare$
    
        \begin{figure}
    	\centering
    	\begin{tikzpicture}[scale=1.4]
    	\draw[fill=gray!30] (-2, 0) circle (1) node {\small$e_2$};
    	\draw[fill=gray!30] (2, 0.75) circle (0.5) node {\tiny$e_1$};
    	\draw[fill=gray!30] (2, -0.75) circle (0.5) node {\tiny$G_2$};
    	\draw (2,0.5) -- (2.8,0.5);
    	\node (K1) at (3,0.5) {\tiny $G_1$};
    	
    	\coordinate (1) at (-1.8,0.15);
    	\coordinate (2) at (-2.2,0.15);
    	\draw (1) -- (2);
    	\fill (1) circle (1pt);
    	\fill (2) circle (1pt);

    	\draw[thick] (2,0.25) -- (2,-0.25) node[midway, right, black] {$f$};
    	\fill (2,0.25) circle (0.05);
    	\fill (2,-0.25) circle (0.05);

        \coordinate (v1) at ($(-2,0)+(48.59:1)$);
        \coordinate (v2) at ($(-2,0)+(34.98:1)$);
        \coordinate (v3) at ($(-2,0)+(23.36:1)$);
        \coordinate (v4) at ($(-2,0)+(-23.36:1)$);
        \coordinate (v5) at ($(-2,0)+(-48.59:1)$);

        \coordinate (w1) at ($(2,0.75)+(180:0.5)$);
        \coordinate (w2) at ($(2,0.75)+(200.7:0.5)$);
        \coordinate (w3) at ($(2,0.75)+(225:0.5)$);
        \coordinate (w4) at ($(2,-0.75)+(135:0.5)$);
        \coordinate (w5) at ($(2,-0.75)+(180:0.5)$);

        \foreach \i in {1,2,3}{
            \draw[red, thick] (v\i)--(w\i);
            \fill (v\i) circle (0.05);
            \fill (w\i) circle (0.05);
        }

        \foreach \i in {4,5}{
            \draw[blue, thick] (v\i)--(w\i);
            \fill (v\i) circle (0.05);
            \fill (w\i) circle (0.05);
        }

    	\coordinate (1) at (1.8,0.9);
    	\coordinate (2) at (2.2,0.9);
    	\draw (1) -- (2);
    	\fill (1) circle (1pt);
    	\fill (2) circle (1pt);
    	
    	\node at (2,-1.5) {$\underbrace{\hspace{1.5cm}}_{G[S]}$};\node[] at (0.25,-1.5) {$\underbrace{\hspace{3cm}}_{\delta(S)}$};
    	\node at (-2,-1.5) {$\underbrace{\hspace{2.5cm}}_{G[V\setminus S]}$};
    	\end{tikzpicture}
    	\caption{Example for $G[S]$ containing a bridge $f$ with $\vert \delta(S,f) \vert=2$. $\delta(S,f)$ is depicted in blue.}
    	\label{bridge2}
    \end{figure}
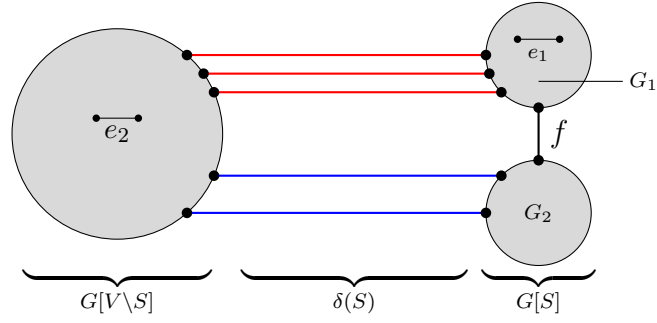

	\medskip
	
	\noindent \textbf{Claim 3}: If neither $e_1$ nor $e_2$ is a bridge in $G[W]$ for $W \in \{S,V \setminus S\}$, and there is a bridge $f$ of $G[W]$, such that $\vert \delta(W,f) \vert \ge 2$, then the considered connectivity cut inequality does not define a facet. 
	
	\medskip
	
	\noindent \textit{Proof.} Without loss of generality we may assume that $f$ is a bridge of $G[S]$. Let $G_1$ and $G_2$ denote the connected component of $G[S]-f$, where $G_1$ contains $e_1$. This situation is depicted in \Cref{bridge2}. We have $\delta(V(G_2))=\{f\} \cup \delta(S,f)$. Suppose $H \subseteq G$ is a $2$-edge-connected subgraph such that $\chi^H \in \F$. Then, one of the following cases occurs:
    \begin{enumerate}
    	\item[$\bullet$] $H$ contains $e_1$ and $e_2$.
    		\begin{enumerate}[(a)]
    			\item If $f \in E(H)$, then $H$ contains exactly one edge of $\delta(S,f)$.
    			\item If $f \notin E(H)$, $H$ contains exactly two edges of $\delta(S) \setminus \delta(S,f)$ and, therefore, $H$ does not contain any edge of $\delta(S,f)$.
    		\end{enumerate}
    	\item[$\bullet$] $H$ contains exactly one edge of $\{e_1,e_2\}$. Then $H$ neither contains an edge of $\delta(S)$ nor the edge $f$.
    \end{enumerate}
     In all cases, $\chi^H$ attains equality in the asymmetric cut inequality
    \begin{align}
        \label{domCCI}
        x_f-x(\delta(S,f))=x_f-x(\delta(V(G_2))\setminus\{f\}) \le 0.
    \end{align}
    Let $\F' \subseteq \TECSP(G)$ denote the face defined by \eqref{domCCI}. In particular, $\F \subseteq \F'$. Moreover, since neither $e_1$ nor $e_2$ is a bridge in $G[S]$ resp. $G[V\setminus S]$, there exists a cycle $C \subseteq G$, that contains exactly one edge of $\delta(S,f)$, one edge of $\delta(S) \setminus \delta(S,f)$ and the edge $f$ but neither $e_1$ nor $e_2$. This means $\chi^C \in \F' \setminus \F$ and, therefore, $\F \subsetneq \F'$. The assumption $\vert \delta(S,f) \vert \ge 2$ implies that $\F'$ is a proper face of $\TECSP(G)$, since for example $\textbf{1}=\chi^G \notin \F'$. Hence, the connectivity cut inequality does not define a facet in this case. \hfill $\blacksquare$
    
    \medskip
    
    \noindent \textbf{Claim 4}: If neither $e_1$ nor $e_2$ is a bridge in $G[S]$ resp. $G[V\setminus S]$, and for all bridges $f$ of $G[W]$ for $W \in \{S,V\setminus S\}$ (if there are any), we have $\vert \delta(W,f) \vert=1$, then the considered connectivity cut inequality defines a facet. 
    
    \medskip

    \noindent \textit{Proof.} Let $\F \subseteq \TECSP(G)$ denote the face defined by the considered connectivity cut inequality, and let $U \subseteq \R^{E(G)}$ be the linear subspace with $\aff(\F)=y+U$ for $y \in \F$. We have to show that $\dim(U)=\vert \CP(G) \vert-1$. It is clear that $\F \subsetneq \TECSP(G)$, since $\textbf{0} \notin \F$, and, hence, $\dim(U) \ge \vert \CP(G) \vert -1$ implies that $\F$ is a facet. For this aim, we will give $\vert \CP(G) \vert -1$ many linearly independent vectors contained in $U$. First, we describe how such a considered cut looks like.
    
    An example for the following can be seen in \Cref{fig: CLaim 4}. Since $\delta(S)$ is an inclusion-wise minimal cut, the graph $G-\delta(S)$ consists of the two connected components $G^{(1)}\coloneqq G[S]$ and $G^{(2)} \coloneqq G[V\setminus S]$ with $e_i \in G^{(i)}$ for $i=1,2$. For $i=1,2$, let $B^{(i)} \subseteq E(G_i)$ denote the set of bridges in $G^{(i)}$; let $G^{(i)}_1,\ldots,G^{(i)}_{n_i}$ denote the connected components of $G^{(i)}-B^{(i)}$, named such that $e_i \in G^{(i)}_1$. We remark that all of these connected components are $2$-edge-connected (including the case that a connected component may just be an isolated vertex), since otherwise there would be more bridges not contained in any $B^{(i)}$. By contracting all connected components of $G^{(i)}-B^{(i)}$ to single vertices, we obtain a tree $T^{(i)}$ that has a vertex $v^{(i)}_j$ for each connected component $G^{(i)}_j$ and edge set $E(T^{(i)})=\{\{v^{(i)}_j,v^{(i)}_{j'}\} \mid \text{there is } b \in B^{(i)} \text{ connecting } G^{(i)}_j \text{ and } G^{(i)}_{j'}\}$. As we did for $B^{(i)}$, we can also interpret the edges of $\delta(S)$ as edges having one end in $V(T^{(1)})$ and the other one in $V(T^{(2)})$ depending on to which components of $G^{(i)}-B^{(i)}$ their incident vertices belong, and we will denote this set by $\overline{\delta(S)}$. Let $i=1,2$. Then, for every leaf $w$ of $T^{(i)}$ there has to be at least one edge $e \in \overline{\delta(S)}$ that is incident to $w$, since otherwise the graph $G$ would not be $2$-edge-connected. Furthermore, for every leaf $w \neq v^{(1)}_1,v^{(2)}_1$ of $T^{(i)}$, there has to be \emph{exactly} one edge $e \in \overline{\delta(S)}$ that is incident to $w$, since otherwise the edges $f \in B^{(i)}$ that correspond to the edges of $E(T^{(i)})$ contained in the unique path $P \subseteq T^{(i)}$ connecting $v^{(i)}_1$ and $w$ would not satisfy the condition $\vert \delta(W_i,f) \vert =1$, for $W_1=S$, and $W_2=V\setminus S$. This also implies that for every non-leaf vertex $w' \in V(T^{(i)})$, $w' \neq v^{(i)}_1$, there is no edge $e \in \overline{\delta(S)}$ incident to $w'$, since otherwise the edges $f \in B^{(i)}$ that correspond to the edges of $E(T^{(i)})$ contained in the unique path $P' \subseteq T^{(i)}$ between $v^{(i)}_1$ and $w'$ would not satisfy the condition $\vert \delta(W_i,f) \vert =1$. Analogously, there can be no vertex $w'' \in V(T^{(i)})\setminus \{v^{(i)}_1\}$ that has degree three or more in $T^{(i)}$. Hence, every vertex $v_i \in V(T^{(i)}) \setminus \{v^{(i)}_1$\} has degree at most two in $T^{(i)}$, the degree one vertices of $T^{(i)}$ are incident to exactly one edge of $\overline{\delta(S)}$, and the degree two vertices of $T^{(i)}$ are not incident to any edge of $\overline{\delta(S)}$. The edges $e \in \overline{\delta(S)}$ then connect $v^{(1)}_1$ or a leaf of $T^{(1)}$ with $v^{(2)}_1$ or a leaf of $T^{(2)}$. In other words, there are $k\coloneqq \vert \delta(S) \vert$ many internally disjoint paths $P_1,\ldots,P_k$ between $v^{(1)}_1$ and $v^{(2)}_1$ in the graph $\overline{G}\coloneqq (V(T^{(1)}\cup T^{(2)}),E(T^{(1)}\cup T^{(2)}) \cup \overline{\delta(S)})$, and each of these paths contains exactly one edge of $\overline{\delta(S)}$. Furthermore, for $\ell \in [k]$, $P_\ell$ corresponds to a subgraph in $G$ by re-substituting, for $i=1,2$, the vertices $v^{(i)}_j$ by the subgraphs $G^{(i)}_j$, the edges $E(T^{(i)})$ by $B^{(i)}$, and $\overline{\delta(S)}$ by $\delta(S)$, and this subgraph will be denoted by $G(P_\ell) \subseteq G$. Note that each $G(P_\ell)$ contains the subgraphs $G^{(1)}_1$ and $G^{(2)}_1$. We remark that the edges in $E(G(P_\ell)) \cap (B^{(1)} \cup B^{(2)} \cup \delta(S))$ define a coparallel class of $G$ for all $\ell \in [k]$, as long as $k \ge 3$. For $k=2$, the edges $B^{(1)} \cup B^{(2)} \cup \delta(S)$ define a coparallel class.
    
    Let $C' \coloneqq E(G(P_1))\cap (B^{(1)}\cup B^{(2)} \cup \delta(S)) \in \CP(G)$ if $k \ge 3$, and $C' \coloneqq B^{(1)}\cup B^{(2)} \cup \delta(S) \in \CP(G)$ for $k=2$. We can now enumerate vectors $\chi^{(C)}$ for all $C \in \CP(G) \setminus \{C'\}$, such that $\chi^{(C)} \in U$ and $(\chi^{(C)})_{C \in \CP(G)\setminus \{C'\}}$ are linearly independent. Note that for two vectors $\chi_1,\chi_2 \in \F$, we have $\chi_1-\chi_2 \in U$. We distinguish several cases depending on the type of the coparallel class.

\begin{figure}
  	\centering
   	\begin{tikzpicture}[scale=1.4]
       	\draw[fill=gray!30] (-4, 0) circle (1) node {\small$e_1$};
        \draw[fill=gray!30] (4, 0) circle (1) node {\small$e_2$};

        \node at (4,0.5) {$G^{(2)}_1$};
        \node at (-4,0.5) {$G^{(1)}_1$};
     
    	\coordinate (1) at (-3.8,-0.15);
    	\coordinate (2) at (-4.2,-0.15);
    	\draw (1) -- (2);
    	\fill (1) circle (1pt);
    	\fill (2) circle (1pt);

        \coordinate (3) at ($(4,0)+(45:0.3)$);
    	\coordinate (4) at ($(4,0)+(-45:0.3)$);
    	\draw (3) -- (4);
    	\fill (3) circle (1pt);
    	\fill (4) circle (1pt);

        % Knoten links für G^(1)_1
        \coordinate (v1) at ($(-4,0)+(60:1)$);
        \coordinate (v2) at ($(-4,0)+(30:1)$);
        \coordinate (v3) at ($(-4,0)+(10:1)$);
        \coordinate (v3a) at ($(-4,0)+(-10:1)$);
        \coordinate (v4) at ($(-4,0)+(-30:1)$);
        \coordinate (v5) at ($(-4,0)+(-60:1)$);

        % Knoten rechts für G^(2)_1
        \coordinate (w1) at ($(4,0)+(120:1)$);
        \coordinate (w2) at ($(4,0)+(150:1)$);
        \coordinate (w3) at ($(4,0)+(170:1)$);
        \coordinate (w3a) at ($(4,0)+(190:1)$);
        \coordinate (w4) at ($(4,0)+(210:1)$);
        \coordinate (w5) at ($(4,0)+(240:1)$);

        % Erster Pfad
        \coordinate (g1) at (-2.5,2.5);
        \coordinate (g2) at (-1,2.5);
        \coordinate (g3) at (1.5,2.5);
    	\draw[fill=gray!30] (g1) circle (0.5) node {\small$G^{(1)}_{2}$};
        \draw[fill=gray!30] (g2) circle (0.5) node {\small$G^{(1)}_{3}$};
        \draw[fill=gray!30] (g3) circle (0.5) node {\small$G^{(2)}_2$};
        
        \coordinate (v6) at ($(g1)+(180:0.5)$);
        \coordinate (v7) at ($(g1)+(0:0.5)$);
        \coordinate (v8) at ($(g2)+(180:0.5)$);
        \coordinate (v9) at ($(g2)+(0:0.5)$);
        \coordinate (v10) at ($(g3)+(180:0.5)$);
        \coordinate (v11) at ($(g3)+(0:0.5)$);

        \draw[blue, thick] (v1) -- (v6);
        \draw[blue, thick] (v7) -- (v8);
        \draw[thick, red] (v9) -- (v10);
        \draw[violet, thick] (v11) -- (w1);
        
        \fill (v6) circle (1.5pt);
        \fill (v7) circle (1.5pt);
        \fill (v8) circle (1.5pt);
        \fill (v9) circle (1.5pt);
        \fill (v10) circle (1.5pt);
        \fill (v11) circle (1.5pt);

        % Zweiter Pfad

        \coordinate (g4) at (-1.5,1);
        \coordinate (g5) at (1.5,1);
    	\draw[fill=gray!30] (g4) circle (0.5) node {\small$G^{(1)}_{4}$};
        \draw[fill=gray!30] (g5) circle (0.5) node {\small$G^{(2)}_{3}$};
        
        \coordinate (v12) at ($(g4)+(180:0.5)$);
        \coordinate (v13) at ($(g4)+(0:0.5)$);
        \coordinate (v14) at ($(g5)+(180:0.5)$);
        \coordinate (v15) at ($(g5)+(0:0.5)$);

        \draw[blue, thick] (v2) -- (v12);
        \draw[red, thick] (v13) -- (v14);
        \draw[violet, thick] (v15) -- (w2);
        
        \fill (v12) circle (1.5pt);
        \fill (v13) circle (1.5pt);
        \fill (v14) circle (1.5pt);
        \fill (v15) circle (1.5pt);

        % dritter Pfad (beide Kanten)
        \draw[red, thick] (v3) -- (w3);
        \draw[red, thick] (v3a) -- (w3a);

        % vierter Pfad
        \coordinate (g6) at (-1.5,-1);
    	\draw[fill=gray!30] (g6) circle (0.5) node {\small$G^{(1)}_5$};
        
        \coordinate (v16) at ($(g6)+(180:0.5)$);
        \coordinate (v17) at ($(g6)+(0:0.5)$);

        \draw[blue, thick] (v4) -- (v16);
        \draw[red, thick] (v17) -- (w4);
        
        \fill (v16) circle (1.5pt);
        \fill (v17) circle (1.5pt);

        % fünfter Pfad
        \coordinate (g7) at (-2.5,-2.5);
        \coordinate (g8) at (-1,-2.5);
        \coordinate (g9) at (1,-2.5);
        \coordinate (g10) at (2.5,-2.5);
    	\draw[fill=gray!30] (g7) circle (0.5) node {\small$G^{(1)}_6$};
        \draw[fill=gray!30] (g8) circle (0.5) node {\small$G^{(1)}_7$};
        \draw[fill=gray!30] (g9) circle (0.5) node {\small$G^{(2)}_4$};
        \draw[fill=gray!30] (g10) circle (0.5) node {\small$G^{(2)}_5$};
        
        \coordinate (v18) at ($(g7)+(180:0.5)$);
        \coordinate (v19) at ($(g7)+(0:0.5)$);
        \coordinate (v20) at ($(g8)+(180:0.5)$);
        \coordinate (v21) at ($(g8)+(0:0.5)$);
        \coordinate (v22) at ($(g9)+(180:0.5)$);
        \coordinate (v23) at ($(g9)+(0:0.5)$);
        \coordinate (v24) at ($(g10)+(180:0.5)$);
        \coordinate (v25) at ($(g10)+(0:0.5)$);

        \draw[blue, thick] (v5) -- (v18);
        \draw[blue, thick] (v19) -- (v20);
        \draw[red, thick] (v21) -- (v22);
        \draw[violet, thick] (v23) -- (v24);
        \draw[violet, thick] (v25) -- (w5);
        
        \fill (v18) circle (1.5pt);
        \fill (v19) circle (1.5pt);
        \fill (v20) circle (1.5pt);
        \fill (v21) circle (1.5pt);
        \fill (v22) circle (1.5pt);
        \fill (v23) circle (1.5pt);
        \fill (v24) circle (1.5pt);
        \fill (v25) circle (1.5pt);

        \fill (v1) circle (1.5pt);
    	\fill (v2) circle (1.5pt);
        \fill (v3) circle (1.5pt);
        \fill (v4) circle (1.5pt);
        \fill (v5) circle (1.5pt);
        \fill (v3a) circle (1.5pt);
        
        \fill (w1) circle (1.5pt);
    	\fill (w2) circle (1.5pt);
        \fill (w3) circle (1.5pt);
        \fill (w4) circle (1.5pt);
        \fill (w5) circle (1.5pt);
        \fill (w3a) circle (1.5pt);
        
        \node at (-2.75,-3.5) {$\underbrace{\hspace{6cm}}_{G[S]}$};
    	\node at (2.75,-3.5) {$\underbrace{\hspace{6cm}}_{G[V\setminus S]}$};
        \node[] at (0,-3.5) {$\underbrace{\hspace{1.5cm}}_{\delta(S)}$};
    	
    	\end{tikzpicture}
    	\caption{Example for the situation in Claim 4. The edges of the cut $\delta(S)$ are colored in red. The edges of $B^{(1)}$ and $B^{(2)}$ are colored in blue and violet, respectively. Each grey-filled circle represents a $2$-edge-connected subgraph $G^{(i)}_j$, for $i=1,2$, and in this example we have $n_1=7$, and $n_2=5$.}
    	\label{fig: CLaim 4}
    \end{figure}
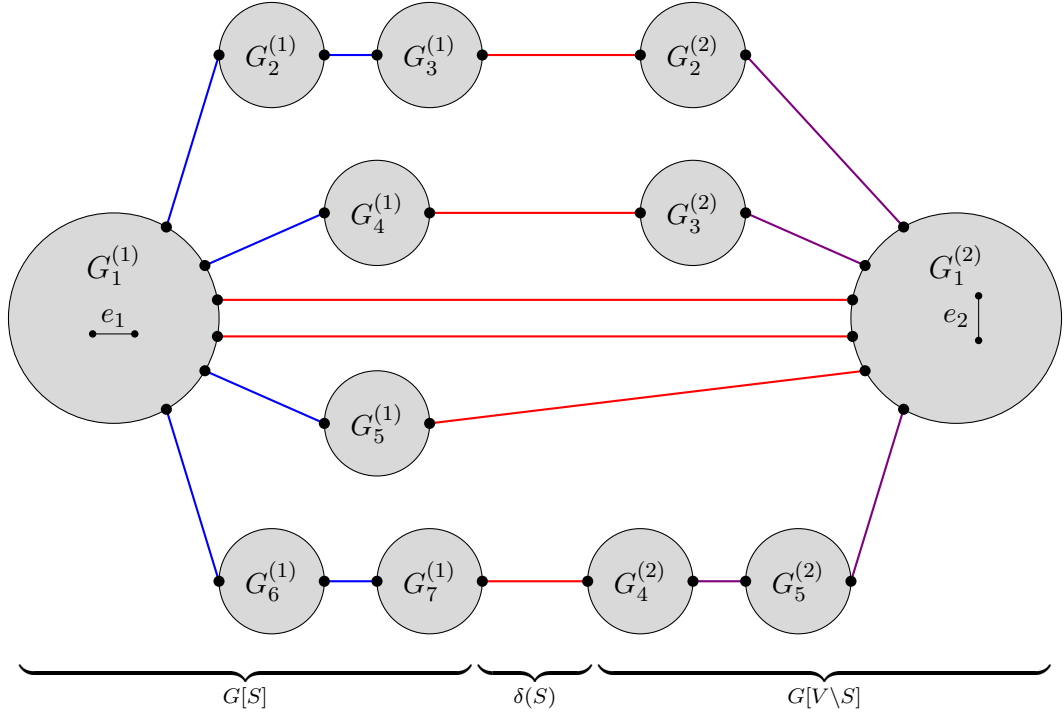

\textit{Case (1): Coparallel classes contained in $G^{(i)}_j$ for $i=1,2$, and $j \in [n_i]\setminus \{1\}$:} Let $C \in \CP_{G^{(i)}_{j}}(G)$ for $i \in \{1,2\}$, and $j \in [n_i] \setminus \{1\}$ (assuming $E(G^{(i)}_{j}) \neq \emptyset$) be a coparallel class. We will define a graph $F$ to help us construct feasible solutions on $\F$: Let $Q \in \{G(P_1),\ldots,G(P_k)\}$ be the subgraph that contains $G^{(i)}_j$, let $Q' \in \{G(P_1),\ldots,G(P_k)\}\setminus\{Q\}$, and let $v_1,v_2 \in V(G^{(i)}_j)$ be the two vertices that are incident to an edge of $\delta(S) \cup B^{(i)}$. Now, let $F$ be the union of $Q$ and $Q'$, after removing the edges $E(G^{(i)}_j)$ and the vertices $V(G^{(i)}_j)\setminus \{v_1,v_2\}$. Consider the set $\mathcal{H} \coloneqq \{H \subseteq G^{(i)}_j \mid F \cup H \text{ is $2$-edge-connected} \}$. We have $\chi^H+\chi^F \in \F$ for all $H \in \mathcal{H}$. Now, let $w \notin V(G^{(i)}_j)$ be a new vertex, $f_1\coloneqq \{v_1,w\}$, and $f_2 \coloneqq \{v_2,w\}$, and consider the graph $\widetilde{G} \coloneqq (V(G^{(i)}_{j}) \cup \{w\},E(G^{(i)}_{j})\cup \{f_1,f_2\})$ (we add two edges here to avoid creating a multi-edge in case $\{v_1,v_2\} \in E(G^{(i)}_j)$). Note that $\CP(\widetilde{G})=\CP_{G^{(i)}_j}(G)\cup\{\{f_1,f_2\}\}$. The incidence vectors $\{\chi^H + \chi^{\{f_1,f_2\}}\in \R^{E(\tilde{G})} \mid H \in \mathcal{H}\}$ are exactly the vertices of the face $\F_{f_1} \subseteq \TECSP(\widetilde{G})$ defined by $x_{f_1} \le 1$. By the proof of \Cref{boxineq1}, there are $\vert \CP(\widetilde{G})\vert-1=\vert \CP_{G^{(i)}_j}(G) \vert$ many linearly independent vectors of the form $\chi^{H_1}-\chi^{H_2}=\chi^{E(H_1)\cup \{f_1,f_2\}}-\chi^{E(H_2)\cup \{f_1,f_2\}} \in \R^{E(\tilde{G})}$ for some $H_1,H_2 \in \mathcal{H}$. Since we have $\chi^{H_1}-\chi^{H_2} \in U$ (now as incidence vectors in $\R^{E(G)}$) for all $H_1,H_2 \in \mathcal{H}$, this implies that we can find $\vert \CP_{G^{(i)}_j}(G) \vert$ many linearly independent vectors in $U$, whose support is a subset of $E(G^{(i)}_{j})$.

We can do the above for any subgraph $G^{(i)}_{j}$, with $i=1,2$ and $j \in [n_i]\setminus\{1\}$, and since the supports of the vectors that correspond to different graphs $G^{(i)}_j$ are disjoint, we find a linearly independent vector for any coparallel class $C\in \CP_{G^{(i)}_{j}}(G)$. We additionally remark that all these vectors are contained in the linear subspace $\{x \in \R^{E(G)} \mid x_e=0 \text{ for all } e \in \delta(S) \cup B^{(1)} \cup B^{(2)} \cup E(G^{(1)}_{1}) \cup E(G^{(1)}_{2})\}$.

\textit{Case (2): Coparallel classes that contain an edge of $\delta(S)$:} For any two distinct subgraphs $H_1,H_2 \in \{G(P_2),\ldots,G(P_k)\}$ (assuming $k \ge 3$), we have $\chi^{G(P_1) \cup H_1},\chi^{H_1 \cup H_2} \in \F$, which implies that $\chi^{H_2}-\chi^{G(P_1)}= \chi^{H_1 \cup H_2}-\chi^{G(P_1) \cup H_1} \in U$. This defines $k-1$ vectors contained in $U$ and projecting these onto the coordinates $\delta(S)\setminus \{e\}$, with $\{e\}=\delta(S) \cap E(G(P_1))$, yields the standard unit basis of $\R^{\delta(S)\setminus\{e\}}$. Hence, these vectors are linearly independent and since the vectors considered in \textit{Case (1)} are all contained in the linear subspace $\{x \in \R^{E(G)} \mid x_e=0 \text{ for all } e \in \delta(S) \}$, which is not true for the new vectors, the vectors of \textit{Case (1)} together with the ones of \textit{Case (2)} are linearly independent. For the next case, we point out that the vectors considered in \textit{Case (2)} are all contained in the linear subspace $\{x \in \R^{E(G)} \mid x_e=0 \text{ for all } e \in E(G^{(1)}_1) \cup E(G^{(2)}_1)\}$. We additionally remark that this is the only case were we encounter a unique coparallel class to which we do not assign a vector in $U$. Furthermore, this also agrees with the case $k=2$, since here the edges in $B^{(1)} \cup B^{(2)} \cup \delta(S)$ define a single coparallel class, and therefore we do not assign a vector of $U$ to this coparallel class.

\begin{figure}
  	\centering
   	\begin{tikzpicture}[scale=1.2]

        \coordinate (1) at (0:3);
        \coordinate (2) at (45:3);
        \coordinate (3) at (90:3);
        \coordinate (4) at (135:3);
        \coordinate (5) at (180:3);
        \coordinate (6) at (225:3);
        \coordinate (7) at (270:3);
        \coordinate (8) at (315:3);

        \coordinate (v1) at ($(1)+(-90:0.5)$);
        \coordinate (v2) at ($(1)+(90:0.5)$);
        \coordinate (v3) at ($(2)+(-45:0.5)$);
        \coordinate (v4) at ($(2)+(135:0.5)$);
        \coordinate (v5) at ($(3)+(0:0.5)$);
        \coordinate (v6) at ($(3)+(180:0.5)$);
        \coordinate (v7) at ($(4)+(45:0.5)$);
        \coordinate (v8) at ($(4)+(225:0.5)$);
        \coordinate (v9) at ($(5)+(90:0.5)$);
        \coordinate (v10) at ($(5)+(-90:0.5)$);
        \coordinate (v11) at ($(6)+(135:0.5)$);
        \coordinate (v12) at ($(6)+(-45:0.5)$);
        \coordinate (v13) at ($(7)+(180:0.5)$);
        \coordinate (v14) at ($(7)+(0:0.5)$);
        \coordinate (v15) at ($(8)+(225:0.5)$);
        \coordinate (v16) at ($(8)+(45:0.5)$);

        \coordinate (m23) at ($0.5*(v2)+0.5*(v3)$);

        \draw[red, line width=0.5mm] (v2) -- (v3);
        \draw[red,line width=0.5mm] (v4) -- (v5);
        \draw[cyan!70,line width=0.5mm] (v6) -- (v7);
        \draw[cyan!70,line width=0.5mm] (v8) -- (v9);
        \draw[cyan!70, line width=0.5mm] (v10) -- (v11);
        \draw[orange!50,line width=0.5mm] (v12) -- (v13);
        \draw[orange!50,line width=0.5mm] (v14) -- (v15);
        \draw[blue,line width=0.5mm] (v16) -- (v1);
        
       	\draw[fill=gray!30] (1) circle (0.5);
        \draw[fill=gray!30] (2) circle (0.5);
        \draw[fill=gray!30] (3) circle (0.5);
        \draw[fill=gray!30] (4) circle (0.5);
        \draw[fill=gray!30] (5) circle (0.5);
        \draw[fill=gray!30] (6) circle (0.5);
        \draw[fill=gray!30] (7) circle (0.5);
        \draw[fill=gray!30] (8) circle (0.5);

        \coordinate (x1) at ($(1)+(0:0.5)$);

        \draw (v2) -- ($(v2)+(45:0.5)$);
        \draw[dotted] ($(v2)+(45:0.5)$) -- ($(v2)+(45:1)$);
        
        \draw (x1) -- ($(x1)+(45:0.5)$);
        \draw[dotted] ($(x1)+(45:0.5)$) -- ($(x1)+(45:1)$);
        
        \draw (x1) -- ($(x1)+(-45:0.5)$);
        \draw[dotted] ($(x1)+(-45:0.5)$) -- ($(x1)+(-45:1)$);
        \fill (x1) circle (2pt);
        \node[below left] (f2) at ($(x1)+(-45:0.5)+(45:0.15)$) {$f_2$};

        \draw ($(3)+(90:0.5)$) -- ($(3)+(90:1)$);
        \draw[dotted] ($(3)+(90:1)$) -- ($(3)+(90:1.5)$);
        \fill ($(3)+(90:0.5)$) circle (2pt);
        \node[right] (f1) at ($(3)+(90:1)$) {$f_1$};

        \draw ($(6)+(185:0.5)$) -- ($(6)+(185:1)$);
        \draw[dotted] ($(6)+(185:1)$) -- ($(6)+(185:1.5)$);
        \fill ($(6)+(185:0.5)$) circle (2pt);

        \draw ($(6)+(245:0.5)$) -- ($(6)+(245:1)$);
        \draw[dotted] ($(6)+(245:1)$) -- ($(6)+(245:1.5)$);
        \fill ($(6)+(245:0.5)$) circle (2pt);
        \node[right] (f4) at ($(6)+(245:1)$) {$f_4$};
        
        \draw (v16) -- ($(v16)+(0:0.5)$);
        \draw[dotted] ($(v16)+(0:0.5)$) -- ($(v16)+(0:1)$);
        \node[below] (f3) at ($(v16)+(0:0.5)$) {$f_3$};

        \foreach \i in {1,...,16}{
            \fill (v\i) circle (2pt);
        }

        %\fill (m23) circle (2pt);

        \coordinate (e1) at ($(2)+(45:0.2)$);
        \coordinate (e2) at ($(2)+(-45:0.2)$);
        \draw (e1) -- (e2);
    	\fill (e1) circle (1pt);
    	\fill (e2) circle (1pt);
        \node (a) at (2) {\tiny$e_1$};

    	\end{tikzpicture}
    	\caption{Example for how a coparallel class of the component $G^{(1)}_1$ looks like. The black edges are edges of $B^{(1)}$ or $\delta(S)$. All edges that have the same color belong to the same coparallel class of $G$; all together, the colored edges form a coparallel class of $G_1^{(1)}$.}
    	\label{fig: CLaim 4 Part 2}
    \end{figure}
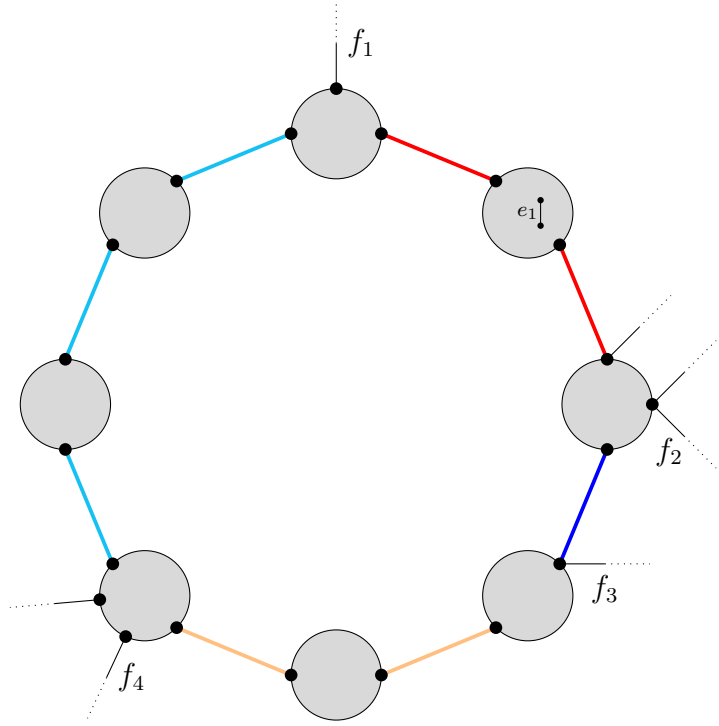

\textit{Case (3): Coparallel classes that are contained in $G^{(i)}_1$ for $i=1,2$:} We will first give a vector for any coparallel class $C \in \CP(G^{(i)}_1)$ for $i=1,2$. We consider all incidence vectors $X \coloneqq \{\chi^H \mid H \subseteq G^{(i)}_1 \,\, 2\text{-edge-connected, } e_i \in E(H)\} \subseteq \F$. With the same arguments as above we can find $\vert \CP(G^{(i)}_1) \vert-1$ many linearly independent vectors of the form $\chi_1-\chi_2 \in U$, with $\chi_1,\chi_2 \in X$. These vectors form a linearly independent family with all the previously considered vectors, since they are supported on a subset of $E(G^{(i)}_1)\setminus\{e_i\}$, whereas for all the so far considered vectors $\chi \in U$ we have $\chi_e=0$ for all $e \in E(G^{(i)}_1)$. This yields $\vert \CP(G_1^{(1)})\vert+\vert \CP(G_1^{(2)})\vert-2$ many new vectors contained in $U$.

Now, pick two distinct graphs $H_1,H_2 \in \{G(P_1),\ldots,G(P_k)\}$. Then $\chi^{H_1 \cup H_2} \in \F$ and this implies $\chi^{H_1 \cup H_2}-\chi^{G^{(1)}_1},\chi^{H_1 \cup H_2}-\chi^{G^{(2)}_1} \in U$. It is clear that these vectors are linearly independent, since projecting these vectors onto the coordinates $\{e_1,e_2\}$ yields the standard unit basis of $\R^{\{e_1,e_2\}}$. This also shows that these two vectors are linearly independent to all previously considered vectors since the latter are all contained in the linear subspace $\{x \in \R^{E(G)} \mid x_{e_1}=x_{e_2}=0\}$. Therefore, we constructed $\vert \CP(G_1^{(1)})\vert+\vert \CP(G_1^{(2)})\vert$ many new linearly independent vectors in \textit{Case (3)}.

Note that $\vert \CP(G^{(1)}_1) \vert \le \vert \CP_{G^{(1)}_1}(G)\vert$ as discussed in the preliminaries. Now, let $C \in \CP(G^{(1)}_1)$ be a coparallel class of $G^{(1)}_1$ that is not a coparallel class of $G$. Then $C=C_1 \dot{\cup} \ldots \dot{\cup} C_t$ is the disjoint union of coparallel classes $C_1,\ldots,C_t \in \CP(G)$. If we contract the connected components of $G^{(1)}_1-C$ to one single vertex each, we obtain a cycle with edges $C$. Moreover, the edges of any coparallel class $C_s$, $s=1,\ldots,t$, form a path in this cycle, which can be seen following the arguments of \Cref{cpc of G-C obs}. See \Cref{fig: CLaim 4 Part 2} for a visualization of the described situation. We can assume that the labeling is chosen such that the edges of $C_s$ and $C_{s+1}$ are next to each other in this cycle for all $s \in [t]$, where $C_{t+1} \coloneqq C_1$. Furthermore, we can assume that either $e_1 \in C_1$ or $e_1$ is contained in a connected component of $G^{(1)}_1-C$ that contains a vertex of an edge of $C_1$. For every $s \in [t]$ there is an edge $f_s \in B^{(1)}\cup\delta(S)$, such that $f_s$ is incident to the vertex in the contracted graph between $C_{s-1}$ and $C_{s}$, setting $C_0 \coloneqq C_t$. We remark that $f_s$ is contained in exactly one graph of $\{G(P_1),\ldots,G(P_k)\}$ and we will denote this graph by $P_{f_s}$. For $s \in [1,t-1]$, let $H_s$ denote the subgraph that contains all connected components of $G^{(1)}_1-C$ that have at least one common vertex with an edge of $C_1 \cup \ldots \cup C_s$, all edges of $C_1 \cup \ldots \cup C_s$, and the edges $(E(P_{f_1}) \cup E(P_{f_{s+1}}))\setminus E(G^{(1)}_1)$ (with induced vertex set). Furthermore, let $H_t\coloneqq P_{f_1} \cup P_{f_t}$. Then $H_s$ is a $2$-edge-connected subgraph and $\chi^{H_s} \in \F$ for all $s \in [t]$. Hence, we have $\chi^{(s)} \coloneqq \chi^{H_{s+1}}-\chi^{H_{s}} \in U$ for all $s \in [t-1]$. Now, choose $g_s \in C_s$ for $s \in [t]$. Projecting $\chi^{(s)}$, $s=1,\ldots,t-1$, onto the coordinates $\{g_2,\ldots,g_t\}$ yields the standard basis of $\R^{\{g_2,\ldots,g_t\}}$, and hence, $\{\chi^{(s)}\}_{s \in [t-1]}$ is linearly independent, while the previously considered vectors are all contained in the linear subspace $U_C \coloneqq \{x \in \R^{E(G)} \mid x_{g_1}=\ldots=x_{g_t}\}$. We thus have $t-1$ additional  vectors contained in $U$. This can analogously be done for every coparallel class of $G^{(2)}_1$ that is not a coparallel class of $G$. We remark that for different coparallel classes $C_1,C_2 \in \CP(G^{(1)}_1) \cup \CP(G^{(2)}_2)$, the newly constructed vectors belonging to $C_1$ are contained in $U_{C_2}$ and the other way around due to \Cref{completelycontained}, and therefore all the constructed vectors are linearly independent.

In total, we constructed $\vert \CP(G) \vert-1$ linearly independent vectors contained in $U$, which completes the proof of Claim 4. \hfill $\blacksquare$

Claims 1 -- 4 establish the proof. \qedhere

\end{proof}

Though it is in general not true that the so far mentioned inequalities describe the whole $2$-edge-connected subgraph polytope $\TECSP(G)$, they suffice to describe all its lattice points.

\begin{theorem}
    \label{lattice points}
    Let $G$ be a $2$-edge-connected graph with edge set $E$. Let \[\P(G) \coloneqq \{x \in [0,1]^{E} \mid x \text{ satisfies all \emph{asymmetric} and \emph{connectivity cut inequalities}}\}.\] Then
    \begin{align*}
        \P(G) \cap \Z^{E} = \{\chi^H \in \{0,1\}^{E} \mid H \subseteq G \text{ is } 2 \text{-edge-connected}\}. 
    \end{align*}
\end{theorem}

\begin{proof}
    It is clear that $\{\chi^H \in \{0,1\}^E \mid H \subseteq G \text{ is } 2 \text{-edge-connected}\} \subseteq \P(G) \cap \Z^{E}$. Let $y \in \P(G)\cap \Z^E$ and let $G[y]$ be the subgraph of $G$ with edge set $E(G[y])=\supp(y)$ and induced vertex set of $E(G[y])$. Hence, $\chi^{G[y]}=y$. The claim follows if we can show that $G[y]$ is $2$-edge-connected. We can assume that $y \neq \textbf{0}$, since otherwise $G[y]$ is the empty graph that is $2$-edge-connected by definition. Suppose first that $G[y]$ is not connected. Then there exist at least two connected components $G_1$ and $G_2$. By definition, $G[y]$ does not contain any isolated vertices and therefore, $E(G_1),E(G_2) \neq \emptyset$. Considering the cut inducing set $S=V(G_1)$, and choosing $e_1 \in E(G_1)$, and $e_2 \in E(G_2)$, we obtain a violated connectivity cut inequality, a contradiction to $y \in \P(G)$. Suppose now that $G[y]$ is connected but not $2$-edge-connected. Then $G[y]$ has a bridge $e \in E(G[y])$. Furthermore, there is a vertex set $\emptyset \subsetneq S \subsetneq V(G)$ such that $\delta(S) \cap E(G[y])=\{e\}$. Then we have
    \begin{align*}
        y_e-y(\delta(S) \setminus\{e\}) = 1,
    \end{align*}
    and hence, $y$ violates an asymmetric cut inequality, again a contradiction to $y \in \P(G)$. This implies that $G[y]$ is $2$-edge-connected.
\end{proof}

Due to the last theorem, we yield an integer linear program for the problem of finding a minimal/maximal weighted $2$-edge-connected subgraph $H \subseteq G$ of a $2$-edge-connected graph~$G$. For more details on that, see \Cref{experiments}.

\subsection{The coparallel class inequalities}

The next type of inequalities we want to present heavily relies on the structure of the coparallel classes. We motivate the idea behind this type of inequalities with the following example. Consider the graph $G$ in \Cref{Example CPCI}. The edges $f_1,f_2$, and $f_3$ form a coparallel class and therefore we have $x_{f_1}=x_{f_2}=x_{f_3}$ for all $x \in \TECSP(G)$. The graph $G-\{f_1,f_2,f_3\}$ consists of three complete graphs $K_3^{(1)},K_3^{(2)}$ and $K_3^{(3)}$. Let $e_i \in E(K_3^{(i)})$ for $i \in [3]$. Then every $2$-edge-connected subgraph $H \subseteq G$ that contains at least two of the edges $e_1,e_2$ or $e_3$ also contains the edges $f_1,f_2$ and $f_3$. This shows that the inequality
\begin{align*}
    x_{e_1}+x_{e_2}+x_{e_3}-2x_{f_1} \le 1
\end{align*}
is a supporting hyperplane for $\TECSP(G)$ that cuts off the otherwise feasible point $y \in \R^{E(G)}$ with $y_{f_1}=y_{f_2}=y_{f_3}=0$ and $y_e=\frac{1}{2}$ for all $e \in E(G) \setminus \{f_1,f_2,f_3\}$.

This can be generalized as follows. Let $G$ be an arbitrary $2$-edge-connected graph. Let $C \in \CP(G)$ be a coparallel class, and let $G_1,\ldots,G_r$ be the connected components of $G-C$ that contain edges and suppose $r \ge 3$. Let $f \in C$ and $e_j \in E(G_j)$ for all $j \in [r]$. Then the inequality
\begin{align*}
    \sum_{j=1}^r x_{e_j}-(r-1)x_f \le 1
\end{align*}
is a supporting hyperplane of $\TECSP(G)$. An inequality of that type will be called a \emph{coparallel class inequality}. Note, that this also makes sense for $r=1$ or $r=2$. But in the first case, one obtains the box inequality $x_{e_1} \le 1$ and in the second case the corresponding inequality is equivalent to the connectivity cut inequality
\begin{align*}
	2x_{e_1}-x_f-x_{f'}+2x_{e_2} \le 2,
\end{align*}
where $f' \in C$ is a second edge of $C$ such that in $G-\{f,f'\}$ the edges $e_1$ and $e_2$ lie in different connected components. In both cases, the considered inequalities define facets due to \Cref{boxineq1} and \Cref{CCI Thm}, respectively. The following theorem shows that the coparallel class inequalities in fact always define facets of $\TECSP(G)$.

\begin{figure}
    \centering
    \begin{tikzpicture}
        \draw[thick] (0,0) -- (3,0);
        \draw[thick] (0,0) -- (60:3);
        \draw[thick] (3,0) -- (60:3);
        \draw[thick] (60:2) -- ($(3,0)+(120:2)$);
        \draw[thick] (1,0) -- (60:1);
        \draw[thick] (2,0) -- ($(3,0)+(120:1)$);
        
        \fill (0,0) circle (2pt);
        \fill (1,0) circle (2pt);
        \fill (2,0) circle (2pt);
        \fill (3,0) circle (2pt);
        \fill (60:1) circle (2pt);
        \fill (60:2) circle (2pt);
        \fill (60:3) circle (2pt);
        \fill ($(3,0)+(120:1)$) circle (2pt);
        \fill ($(3,0)+(120:2)$) circle (2pt);
        \coordinate[label=left:$f_1$] (D) at (60:1.5);
        \coordinate[label=below:$f_2$] (E) at (1.5,0);
        \coordinate[label=right:$f_3$] (F) at ($(3,0)+(120:1.5)$);
    \end{tikzpicture}
    \caption{The graph used in the initial motivating example for coparallel class inequalities.}
    \label{Example CPCI}
\end{figure}
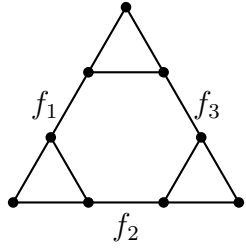

\begin{theorem}
    Let $G$ be $2$-edge-connected, let $C \in \CP(G)$ be a coparallel class, and let $G_1,\ldots,G_r$ be the connected components of $G-C$ that contain edges. Let $f \in C$ and $e_j \in E(G_j)$ for $j \in [r]$. Then the coparallel class inequality
    \begin{align*}
        \sum_{j=1}^r x_{e_j}-(r-1)x_f \le 1
    \end{align*}
    defines a facet of $\TECSP(G)$.
\end{theorem}

\begin{proof}
Let $\F \subseteq \TECSP(G)$ denote the face that is defined by the considered coparallel class inequality, and let $U \subseteq \R^{E(G)}$ be the linear subspace with $\aff(\F)=y+U$ for $y \in \F$. Since $\textbf{0} \notin \F$, we know that $\F$ is a proper face of $\TECSP(G)$. We will construct $\vert \CP(G) \vert -1$ many linearly independent vectors contained in $U$, which implies that $\F$ is a facet. For $i \in [r]$, we consider the vectors $X_i \coloneqq \{\chi^H \mid H \subseteq G_i, \, e_i \in E(H)\}$. After projecting them onto the coordinates $E(G_i)$, these are exactly the vertices of the face $\F_{e_i} \subseteq \TECSP(G_i)$ of $\TECSP(G_i)$ defined by $x_{e_i} \le 1$, . By the proof of \Cref{boxineq1}, this implies that there are $\vert \CP(G_i) \vert -1$ many linearly independent vectors of the form $\chi_1-\chi_2$, with $\chi_1, \chi_2 \in X_i$, and we have $\chi_1-\chi_2 \in U$. We remark that $\supp(\chi_1-\chi_2) \subseteq E(G_i)\setminus \{e_i\}$. Since we can construct this kind of vectors for any $i \in [r]$, and the vectors corresponding to different components $G_i \neq G_j$ have disjoint support, all the considered vectors are linearly independent. In total we have $\vert \CP(G_1) \vert +\ldots+\vert \CP(G_r)\vert -r$ many vectors so far. Now, let $C' \in \CP(G_i)$ be a coparallel class of $G_i$, that is not a coparallel class of $G$. We know by \Cref{cpc of G-C obs} that $C'=C_1 \dot{\cup} C_2$, where $C_1$ and $C_2$ are coparallel classes of $G_i$. Let $G^{(1)}_1,\ldots,G^{(1)}_{\ell_1}$ ($G^{(2)}_1,\ldots,G^{(2)}_{\ell_2}$) denote the connected components of $G_i-C'$ that contain a vertex of an edge of $C_1$ ($C_2$, respectively). We remark that there are one or two components that contain a vertex of both an edge of $C_1$ and an edge of $C_2$; they are thus named twice. For $j=1,2$, we define the subgraph $H_j$ containing (a) all connected components of $G-C$ except $G_i$, (b) the components $G^{(j)}_1,\ldots,G^{(j)}_{\ell_j}$, and (c) all edges of $C_j$. Then $H_1$ and $H_2$ are $2$-edge-connected subgraphs and at least one of them contains the edge $e_i$. Hence, $\chi^{H_1}$ or $\chi^{H_2}$ is an incidence vector $\chi$ contained in $\F$, and $\chi^G-\chi \in U$. Let $e \in C_1$, and $e' \in C_2$, then all the so far considered vectors are contained in the linear subspace $U' \coloneqq \{x \in \R^{E(G)} \mid x_e=x_{e'}\}$, but $\chi^G-\chi \notin U'$, which shows that $\chi^G-\chi$ is linearly independent to all the previously considered vectors.

We can do this for every coparallel class $C' \in \CP(G_1) \cup \ldots \cup \CP(G_r)$ that is not a coparallel class of $G$, and all newly constructed vectors are linearly independent to all the previously considered vectors by the same argument. Thus, there are $\vert \CP(G) \vert-r-1$ many affinely independent vectors contained in $\F$, since $\vert \CP(G) \vert=\vert \CP_{G_1}(G) \vert + \ldots +\vert \CP_{G_r}(G)\vert+1$. Furthermore, we have $Y\coloneqq \{\chi^G-\chi^{G_i} \mid i \in [r]\} \subseteq U$. These vectors are linearly independent, as one can see easily by projecting them onto the coordinates $\{e_1,\ldots,e_r\}$. All the so far considered vectors are contained in the linear subspace $\{x \in \R^{E(G)}\mid x_{e_1}=\ldots=x_{e_r}=0\}$ which is not true for the vectors in $Y$. This shows that $\dim(U) \ge \vert \CP(G) \vert -1$, which finishes the proof.

\end{proof}

\subsection{The odd star inequalities}

The last type of inequalities will be especially defined for complete graphs $K_n$ on $n$ vertices. For example, consider the graph $K_4$ that can be seen in \Cref{K4}. Let $e_1$, $e_2$ and $e_3$ be the edges incident to some vertex $v$ and be blue, and let the other edges $f_1$, $f_2$, and $f_3$ be red. Every $2$-edge-connected subgraph $H \subseteq K_4$ that contains at least one edge incident to $v$ has to contain at least one other blue edge and one red edge. If all three edges $e_1$, $e_2$, and $e_3$ are contained, then $H$ has to contain at least two red edges. This shows that the inequality
\begin{align}
	\label{example OSI}
    x_{e_1}+x_{e_2}+x_{e_3}-x_{f_1}-x_{f_2}-x_{f_3} \le 1
\end{align}
defines a supporting hyperplane for $K_4$, which cuts off the otherwise feasible point $y \in \R^{E(G)}$ with $y_{e_i}=1$ and $y_{f_i}=\frac{1}{2}$ for all $i \in [3]$.

Our goal is to generalize inequality \eqref{example OSI} to complete graphs of arbitrary size $n$ depending on the parity of $n$. First, assume that $n$ is even. Pick any vertex $v \in V(K_n)$ and consider the set of edges $F \coloneqq \delta(\{v\})$ incident to $v$. A $2$-edge-connected subgraph $H \subseteq K_n$ cannot contain only one edge of $F$. If $H$ contains $k \ge 2$ edges of $F$, $H$ also has to contain $\lceil \frac{k}{2} \rceil$ edges of $E(K_n)\setminus F$ in order to ensure that $H$ is $2$-edge-connected. Hence, if $H$ contains $k \ge 2$ edges of $\delta(\{v\})$, we have
\begin{align*}
	\chi^H(F)-\chi^H(E(K_n)\setminus F) \le \lfloor \frac{k}{2} \rfloor \le \lfloor\frac{n-1}{2}\rfloor.
\end{align*}
Thus, the inequality
\begin{align*}
    x(\delta(\{v\}))-x(E(K_n-v)) \le \frac{n-2}{2}
\end{align*}
defines a supporting hyperplane for $\TECSP(K_n)$ for $n$ even and an inequality of that type will be called an \textit{odd star inequality}, as $\delta(\{v\})$ induces a star with an odd number of edges.

Suppose now that $n \ge 5$ is odd. Let $v \in V(K_n)$ and $h=\{w_1,w_2\} \in E(K_n-v)$. Let $f \coloneqq \{v,w_1\} \in E$. Consider the inequality
\begin{align*}
	x(\delta(\{v\})\setminus \{f\})-x(E(K_n-v) \setminus \{h\}))-x_f \le \frac{n-3}{2}.
\end{align*}
Note that this inequality restricted to the edges $E(K_n-w_1)$ defines an odd star inequality for $\TECSP(K_{n-1})$, where $K_{n-1}$ is the complete graph on the vertices $V(K_n) \setminus \{w_1\}$. Similar as for $n$ even, one can see that this defines a supporting hyperplane for $\TECSP(K_n)$ if $n$ is odd and an inequality of that type will also be called an \textit{odd star inequality}. In fact, all odd star inequalities define facets as the following theorem shows.

\begin{figure}
    \begin{subfigure}[b]{0.45 \textwidth}
    	\centering
    	\begin{tikzpicture}[scale=1.3]
        \draw[white] (0,2) -- (0,-2);
        
        \coordinate (A) at (0,0);
        \coordinate (B) at (1.5,1);
        \coordinate (C) at (1.5,0);
        \coordinate (D) at (1.5,-1);
        
        \draw[blue, thick] (A) -- (B);
        \draw[blue, thick] (A) -- (C);
        \draw[blue, thick] (A) -- (D);
        \node[left] at (A) {$v\,$};
        \draw[red, bend left, thick] (B) to (C);
        \draw[red, bend left=50, thick] (B) to node[right] {} (D);
        \draw[red, bend left, thick] (C) to (D);
        \draw[fill=black] (A) circle (0.1);
        \draw[fill=black] (B) circle (0.1);
        \draw[fill=black] (C) circle (0.1);
        \draw[fill=black] (D) circle (0.1);

    \end{tikzpicture}
    \caption{$K_4$}
    \label{K4}
\end{subfigure}
\hfill
\begin{subfigure}[b]{0.45 \textwidth}
\centering
\begin{tikzpicture}[scale=1.3]
	\draw[white] (0,2) -- (0,-2);
	
    \coordinate (A) at (0,0);
    \coordinate (B) at (1.5,1.5);
    \coordinate (C) at (1.5,0.5);
    \coordinate (D) at (1.5,-0.5);
    \coordinate (E) at (1.5,-1.5);
    
    \draw[blue, thick] (A) -- (B);
    \draw[blue, thick] (A) -- (C);
    \draw[blue, thick] (A) -- (D);
    \draw[red, thick] (A) -- (E);
    \node[left] at (A) {$v\,$};
    \node[right] at (D) {\small$w_2$};
    \node[right] at (E) {\small$w_1$};
    \node at ($0.5*(D)+0.5*(E)$) {\small$h$\,};
    \node[below left] at ($0.5*(A)+0.5*(E)$) {\small$f$};
    \draw[red, bend left, thick] (B) to (C);
    \draw[red, bend left=60, thick] (B) to node[right] {} (D);
    \draw[red, bend left=60, thick] (C) to node[right] {} (E);
    \draw[red, bend left=60, thick] (B) to node[right] {} (E);
    \draw[red, bend left, thick] (C) to (D);
    \draw[bend left, thick, label=$h$] (D) to (E);
    \draw[fill=black] (A) circle (0.1);
    \draw[fill=black] (B) circle (0.1);
    \draw[fill=black] (C) circle (0.1);
    \draw[fill=black] (D) circle (0.1);
    \draw[fill=black] (E) circle (0.1);
    \end{tikzpicture}
    \caption{$K_5$}
    \label{K5}
\end{subfigure}
    \caption{Odd star inequality for complete graphs $K_4$ and $K_5$ both with right-hand side $1$. Blue edges and red edges have coefficients $+1$ and $-1$, respectively.}
    \label{Example K4 K5}
\end{figure}

\begin{theorem}
	Let $n \in \N$, $n \ge 4$.
	\begin{enumerate}[(i)]
		\item Let $n$ be even and $v \in V(K_n)$. Then the odd star inequality
		\begin{align*}
			x(\delta(\{v\}))-x(E(K_n-v)) \le \frac{n-2}{2}
		\end{align*}
		defines a facet of $\TECSP(K_n)$.
		\item Let $n$ be odd, $v \in V(K_n)$, and $h=\{w_1,w_2\} \in E(K_n-v)$. Let $f \coloneqq \{v,w_1\}$. Then the odd star inequality
		\begin{align*}
		x(\delta(\{v\})\setminus \{f\})-x(E(K_n-v) \setminus \{h\}))-x_f \le \frac{n-3}{2}
		\end{align*}
		defines a facet of $\TECSP(K_n)$.
	\end{enumerate}
\end{theorem}

\begin{proof}
	Since $K_n$ has neither $1$- nor $2$-cuts for $n \ge 4$, $\TECSP(K_n)$ is full dimensional. Let ${\F \subseteq \TECSP(K_n)}$ be the face defined by the considered odd star inequality and let ${U \subseteq \R^{E(K_n)}}$ be the linear subspace with $\aff(\F)=y+U$ for $y \in \F$. We have to show that $\dim(U)=\vert E(K_n) \vert -1$. Since $\chi^{K_n} \notin \F$, the considered odd star inequality defines a proper face of $\TECSP(G)$, and we have $\dim(U) \le \vert E(K_n) \vert-1$.

\smallskip
\noindent \textit{Case (i)}. Suppose $n$ is even. In the following, we construct two different kind of vectors contained in $U$ that span a subspace of dimension at least $\vert E(K_n) \vert-1$. Let ${f_1 \neq f_2 \in E(K_n-v)}$ with $f_1 \cap f_2=\{w\}$ for some $w \in V(K_n) \setminus \{v\}$. Let $u_1, \ldots, u_{n-2}$ be the vertices $V(K_n)\setminus\{v,w\}$. Let $F' \coloneqq \{\{u_1,u_2\},\{u_3,u_4\},\ldots,\{u_{n-3},u_{n-2}\}\} \subseteq E(K_n)$. The three subgraphs
\begin{align*}
    H& \coloneqq (V(K_n)\setminus \{w\},(\delta(\{v\})\setminus \{\{v,w\}\})\cup F'), \\
    H_1& \coloneqq H+w+\{v,w\}+f_1= (V(K_n),E(H) \cup \{\{v,w\},f_1\}) \text{, and} \\
    H_2& \coloneqq H+w+\{v,w\}+f_2= (V(K_n),E(H)\cup \{\{v,w\},f_2\})
\end{align*}
     are $2$-edge-connected. Furthermore, we have
\begin{align*}
    \chi^{H}(\delta(\{v\}))-\chi^{H}(E(K_n-v))=(n-2)-\frac{n-2}{2}=\frac{n-2}{2},
\end{align*}
which shows $\chi^H \in \F$. Since $H_1$ and $H_2$ both contain one edge more of $\delta(\{v\})$ and $E(K_n-v)$ than $H$, we also have $\chi^{H_1},\chi^{H_2} \in \F$. Hence, $\chi^{\{f_1\}}-\chi^{\{f_2\}}=\chi^{H_1}-\chi^{H_2} \in U$ for any two adjacent edges $f_1,f_2 \in E(K_n-v)$, which leads to the first type of vectors we are considering: Let $f'_1,\ldots,f'_k$, $k\coloneqq \vert E(K_n) \vert-n+1$, denote the edges of $K_n-v$, and we assume w.l.o.g. that the labeling is chosen such that $f'_i$ and $f'_{i+1}$ are adjacent for all $i \in [k-1]$ (which is trivial to exist). We have $S_1 \coloneqq
\{\chi^{\{f'_i\}}-\chi^{\{f'_{i+1}\}} \mid i \in [k-1]\} \subseteq U$. Since all vectors contained in $S_1$ are linearly independent, $S_1$ spans a linear subspace $U_1 \subseteq U$ of dimension $k-1=\vert E(K_n) \vert-n$. We remark that $U_1 \subseteq \{x \in \R^{E(K_n)} \mid x_{\{v,w\}}=0 \text{ for } w \in V(K_n) \setminus \{v\}\}.$ %\text{, and }x(E(K_n-v))=0\}.

We also have $\chi^{\{\{v,w\}\}}+\chi^{\{f_1\}}=\chi^{H_1}-\chi^{H} \in U$, which yields the second type of vectors we are considering. For each $w \in V(K_n) \setminus \{v\}$ pick one vector $y^{(w)} \in \R^{E(K_n)}$ of that type with $y^{(w)}_{\{v,w\}}=1$, i.e., $y^{(w)}=\chi^{\{\{v,w\}\}}+\chi^{\{f'\}}$ for some $f' \in E(K_n-v)$, and let $U_2 \subseteq U$ be the span of the vectors $S_2 \coloneqq \{y^{(w)} \mid{w \in V(K_n)\setminus \{v\}}\}$. Since all the vectors of $S_2$ are linearly independent, we have $\dim(U_2) =n-1$. Since the only vector in $\chi \in U_2$ with $\chi_{\{v,w\}}=0$ for all $w \in V(K_n)\setminus \{v\}$ is the all zeroes vector $\textbf{0}$, we see that $U_1 \cap U_2=\{\textbf{0}\}$. This implies $\dim(U) \ge \dim(U_1) + \dim(U_2)=\vert E(K_n)\vert-n+n-1=\vert E(K_n)\vert-1$, and this shows \textit{(i)}.

\smallskip
\noindent \textit{Case (ii)}. Suppose now that $n$ is odd. As in the proof of \textit{(i)}, we again construct sufficiently many linearly independent vectors contained in $U$.
		
The first class of vectors is constructed as follows. The graph $K \coloneqq K_n-w_1 \subseteq K_n$ is a complete graph on $n-1$ many vertices. As already mentioned above, we can consider the odd star inequality for $K$:
\begin{align*}
    x(\delta(\{v\}))-x(E(K-v)) \le \frac{n-3}{2}.
\end{align*}
Let $\F' \subseteq \TECSP(K)$ denote the corresponding face. Then it is easy to see that for $x \in \F'$, we have $\tilde{x} \in \F$, where $\tilde{x}$ is $x$ lifted to $\R^{E(K_n)}$ by adding zeroes at the missing coordinates. By the proof of \textit{(i)} w.r.t. $K$, this implies that for each pair of adjacent edges $f_1 \neq f_2 \in E(K_n-v-w_1)$ we have $\chi^{\{f_1\}}-\chi^{\{f_2\}} \in U$, and for each vertex $w \in V(K_n) \setminus \{v,w_1\}$ we have $\chi^{\{\{v,w\}\}}+\chi^{\{f'\}} \in U$ for some $f' \in E(K_n-v-w_1)$. Again by the proof of \textit{(i)}, the span of these vectors, denoted by $U'_1 \subseteq U$, has dimension $\vert E(K_n-w_1) \vert-1=\vert E(K_n)\vert-n$. We remark that $U'_1 \subseteq \{x \in \R^{E(K_n)} \mid x_{f'}=0 \text{ for all } f' \in \delta(\{w_1\})\}$.

For the second type of vectors, let $f'=\{w_1,w\} \in \delta(\{w_1\})\setminus\{f,h\}$ for some $w \in V(K_n) \setminus \{v,w_1,w_2\}$. Let $u_1,\ldots,u_{n-4}$ be the vertices $V(K_n) \setminus \{v,w_1,w_2,w\}$ and let $H_1^{(w)} \subseteq K_n$ be the $2$-edge-connected subgraph with vertex set $V(H_1^{(w)})=V(K_n)\setminus\{w_1,u_{n-4}\}$ and edge set \begin{align*}
    E(H_1^{(w)})=(\delta(\{v\})\setminus\{f,\{v,u_{n-4}\}\}) \cup \{\{w,w_2\},\{u_1,u_2\},\ldots,\{u_{n-6},u_{n-5}\}\} \subseteq E(K_n-w_1).
\end{align*}
Consider the graph $H_2^{(w)}$ with vertices $V(H_2^{(w)})=V(H_1^{(w)}) \cup \{w_1\}$ and edges of $E(H_2^{(w)})=E(H_1^{(w)})\setminus \{\{w_2,w\}\}\cup\{h,f'\}$. It is immediate that $\chi^{H_1^{(w)}},\chi^{H_2^{(w)}} \in \F$. This means ${\chi^{H_2^{(w)}}-\chi^{H_1^{(w)}}}$ ${=\chi^{\{f'\}}+\chi^{\{h\}}-\chi^{\{\{w_2,w\}\}} \in U}$ for all $f'=\{w_1,w\}\in \delta(\{w_1\})\setminus \{f,h\}$. Hence, $S'_2 \coloneqq \{\chi^{\{f'\}}+\chi^{\{h\}}-\chi^{\{\{w_2,w\}\}} \mid f'=\{w_1,w\} \in \delta(\{w_1\}) \setminus \{f,h\}\} \subseteq U$. Projecting the vectors in $S'_2$ onto the coordinates $\delta(\{w_1\})\setminus \{f,h\}$ yields the standard unit vectors of $\R^{\delta(\{w_1\})\setminus \{f,h\}}$. Hence, $S'_2$ spans an $(n-3)$-dimensional linear subspace $U'_2$. Clearly, the only vector $\chi \in U'_2$ with $\chi_{\bar{f}}=0$ for all $\bar{f} \in \delta(\{w_1\})$ is the all-zeros vector $\textbf{0}$, which means $U'_1 \cap U'_2=\{\textbf{0}\}$. This implies $\dim(U) \ge \dim(U'_1)+\dim(U'_2) = \vert E(K_n) \vert -3$.

For the last two missing vectors, let $u_1,\ldots,u_{n-3}$ denote the vertices $V(K_n)\setminus\{v,w_1,w_2\}$. Let $H_1' \subseteq K_n$ be the $2$-edge-connected subgraph with vertices $V(H_1')=V(K_n) \setminus \{w_1,w_2\}$ and edges $E(H_1')=\delta(\{v\})\setminus\{f,\{v,w_2\}\}\cup\{\{u_1,u_2\},\ldots, \{u_{n-4},u_{n-3}\}\}$. Let the subgraph $H_2'$ contain all vertices $V(K_n)$, all edges of $H_1'$, and additionally the edges $f$, $h$, and $\{v,w_2\}$. $H_2'$ is $2$-edge-connected, too, and we have $\chi^{H_1'},\chi^{H_2'} \in \F$. Hence, $\chi' \coloneqq \chi^{H_2'}-\chi^{H_1'}=\chi^{\{f\}}+\chi^{\{h\}}+\chi^{\{\{v,w_2\}\}} \in U$. Since $U'_1,U'_2 \subseteq \{x \in \R^{E(G)} \mid x_f=0\}$, but $\chi'_f=1$, the dimension of $U$ is at least $\vert E(K_n) \vert -2$.

For the last vector, let $H_3' \subseteq K_n$ be the $2$-edge-connected subgraph with vertex set $V(H_3')=V(K_n)$ and edge set
\begin{align*}
    E(H_3')=(\delta(\{v\})\setminus \{f\} )\cup \{h,\{w_1,u_1\},\{w_1,u_2\},\{u_3,u_4\},\ldots,\{u_{n-4},u_{n-3}\}\}.
\end{align*}
Furthermore, let $H_4'$ be the subgraph, with vertex set $V(H_4')=V(K_n)\setminus\{w_1,w_2\}$, and edge set
\begin{align*}
    E(H_4')=(\delta(\{v\})\setminus\{f,\{v,w_2\}) \cup \{\{u_1,u_2\},\ldots,\{u_{n-4},u_{n-3}\}\}.
\end{align*}
It is easy to check that $\chi^{H_3'},\chi^{H_4'} \in \F$. Hence, $\chi^{H_3'}-\chi^{H_4'}=\chi^{\{h\}}+\chi^{\{\{w_1,u_1\}\}}+\chi^{\{\{w_1,u_2\}\}}+\chi^{\{v,w_2\}}-\chi^{\{\{u_1,u_2\}\}} \in U$. All of the previously considered vectors are contained in the linear subspace $\{x \in \R^{E(G)} \mid x_h-x(\delta(\{w_1\}) \setminus \{h\})=0\}$, which is not true for $\chi^{H_3'}-\chi^{H_4'}$. Hence, we have $\dim(U) \ge \vert E(K_n)\vert -1$, which finishes the proof. \qedhere
\end{proof}

\section{Computing 2-edge-connected subgraphs}
\label{experiments}
%einleitung: Praxis

We can now discuss how to transfer the theoretical results to an implementation for finding a maximum weight $2$-edge-connected subgraph. As motivated in the introduction, and as done throughout the paper, we assume that the given graph $G$ is $2$-edge-connected. 

\subsection{Basic model}

\label{subsection basic model}
%Wegen Th. erstes Modell ohne extra Ungleichungen
%Was ist branch und cut? Was ist Separieren?
%Wie Werden Ungleichungen separiert?

By \Cref{lattice points}, the following defines an integer linear program (ILP) to find a maximum weight $2$-edge-connected subgraph $H \subseteq G$ of a $2$-edge-connected graph $G$ with edge weights $w \colon E(G) \to \R$:

\begin{center}
\begin{tabular}{r c l l}
	$\displaystyle \max \sum_{e \in E(G)} w(e) \cdot x_e$, & & & such that \\
	$x_f-x(\delta(S) \setminus \{f\})$ & $\le$ & $0$, & for all $\emptyset \subsetneq S \subsetneq V(G)$, $f \in \delta(S)$,\\
	$2x_{e_1}-x(\delta(S))+2x_{e_2}$ & $\le$ & $2$, & for all $\emptyset \subsetneq S \subsetneq V(G)$, $e_1 \in E(G[S]),e_2 \in E(G[V \setminus S])$,\\
 $x$ & $\in$ & $\{0,1\}^{E(G)}$. &
\end{tabular}
\end{center}

Such an ILP is usually solved with a \emph{Branch-and-Cut} algorithm: We replace the condition $x \in \{0,1\}^{E(G)}$ by $x \in [0,1]^{E(G)}$ to obtain the easier to solve \emph{LP-relaxation}. Moreover, we start without any other inequality of the ILP. Then we use a linear program (LP) solver, to get a solution $\bar{x} \in [0,1]^{E(G)}$ of this LP. Now we have to solve the \emph{separation problem}: We have to check if $\bar{x}$ satisfies all inequalities of our full LP-relaxation, or find inequalities (at least one) violated by $\bar{x}$. In the latter case, we will add (some of) these inequalities to our current model. We repeat this process of solving and separating (cutting) until there are no more violated inequalities. In that way, we obtain a solution of the full LP-relaxation of our ILP. If $\bar{x} \in \{0,1\}$, we found an optimal solution for our ILP. Otherwise, there is at least one coordinate where $\bar{x}$ has a fractional value. In that case, we split the problem into two subproblems, recursively solving their linear program, with the variable fixing $\bar{x}_e=0$ and $\bar{x}_e=1$, respectively.

To use this algorithm, we have to solve the separation problem for the asymmetric cut inequalities and the connectivity cut inequalities: decide whether a given point $\bar{x} \in [0,1]^{E(G)}$ satisfies all asymmetric cut inequalities (connectivity cut inequalities, resp.) and if not, identify at least one violated inequality. For the asymmetric cut inequalities this can be done as follows. For each edge $e=\{s,t\} \in E(G)$, compute a minimum $s$-$t$-cut $F_e \subseteq E(G)\setminus \{e\}$ in the graph $G-e$ with edge weights $w'(e') \coloneqq \bar{x}_{e'}$ for all $e' \in E(G) \setminus \{e\}$. Then, $F_e \cup \{e\}$ is an $s$-$t$-cut in $G$ and
\begin{align*}
    x_e-x(F_e) \le 0
\end{align*}
defines an asymmetric cut inequality. If $\bar{x}$ violates this inequality, we identified a maximally violated asymmetric constraint among these with positive $x_e$-coefficient. If this process does not yield a violated constraint for any edge $e \in E(G)$, all asymmetric cut inequalities are satisfied.

For the connectivity cut inequalities we can proceed similarly: For every pair of non-adjacent edges $e_1,e_2 \in E(G)$, compute a minimum $s$-$t$-cut $F \subseteq E(G)$, where $s$ is a vertex incident to $e_1$ and $t$ is a vertex incident to $e_2$ with edge weights $w'(e')\coloneqq x_{e'}$ for all $e' \in E(G) \setminus \{e_1,e_2\}$ and $w'(e_1) \coloneqq w'(e_2) \coloneqq 2\vert E(G) \vert$. By the choice of the weights we have $e_1,e_2 \notin F$ and therefore,
\begin{align*}
    2x_{e_1}-x(F)+2x_{e_2} \le 2
\end{align*}
defines a connectivity cut inequality of $G$. Again, either this yields a most violated connectivity cut constraint for some $e_1$, $e_2$, or all such constraints are satisfied.

\subsection{Strengthening the basic model}

%1. Coparallel class plus wie separieren
%2. odd star

We can strengthen the above ILP by adding the coparallel class inequalities, which also requires us to have a separation routine. Again, let $\bar{x} \in [0,1]^{E(G)}$ be the current solution, possibly not satisfying all coparallel class inequalities. For each coparallel class $C \in \CP(G)$, let $G_1,\ldots,G_r$ denote the connected components of $G-C$ that contain edges. If $r \ge 3$, let $e_i \coloneqq \mathrm{arg\,max}_{e \in E(G_i)} \{\bar{x}_e\}$, for $i \in [r]$, breaking ties arbitrarily. Then 
\begin{align*}
    \sum_{i=1}^r x_{e_i} - (r-1)x_f \le 1
\end{align*}
for $f \in C$ defines a coparallel class inequality of $G$. If $\bar{x}$ violates any coparallel class inequality, then also one of these constructed.

If $G$ is a complete graph on $n \ge 4$ vertices, we can strengthen our model by the odd star inequalities. There are no coparallel class inequalities, but since there are only $n$ and $n(n-1)(n-2)$ odd star inequalities for $n$ even and $n$ odd, respectively, we do not need a separation routine, but can add them directly to our model in the beginning. 

\subsection{Experimental setup}

In order to test whether the coparallel class inequalities are (can be) beneficial in practice, we want to use graphs that have rather large coparallel classes. To this end, we construct two different classes of graphs. The first class $\mathcal{G}(n,k,\alpha)$ are sparsified $k$-nearest-neighbor graphs generated as follows. We pick $n$ random points in the plane which are the vertices of the graph. For each point we connect it with its $k$ nearest neighbors and check if we obtain a $2$-edge-connected graph. Afterwards, we pick $\alpha$-percent of the edges of the so far constructed graph randomly. In a random order we delete each of the chosen edges if the deletion still leaves a $2$-edge-connected graph, otherwise we just continue with the next edge. We then pick a random edge weight for the resulting graph by assigning integer values between $-5$ and $6$ uniformly distributed and independently for each edge. For the experiments we choose $n=150$, $k \in \{4,5\}$, and $\alpha \in \{70\%,80\%,90\%\}$. For each of the possible combinations, we generate $15$ graphs.

For the second type of graphs, which we will call $K_n$-cycles, we pick a number $\ell$. We then generate $\ell$ complete graphs $K^{(1)}_{n_1}, \ldots, K^{(\ell)}_{n_\ell}$ of size $n_i$, where $n_i$ is an integer between $3$ and $7$ chosen independently for each $i$ uniformly at random. We then partition the $\ell$ graphs into groups of size $3$ to $7$. Let $G_1, \ldots, G_k$ be the graphs of one group. We consider the graph $G_1 \cup \ldots \cup G_k$ and for $i \in [k]$ we add edges $e_i=\{v_i,w_i\}$, where $v_i \in V(G_i)$ and $w_i \in V(G_{i+1 \bmod k})$. This yields a graph where the new added edges form a coparallel class and the connected components of this graph without the coparallel class are exactly the graphs of this group. Now we have a collection of graphs $G'_1,\ldots,G'_{\ell'}$ with $\ell' < \ell$. For these graphs we repeat this procedure above until there is only one graph left. Finally, for the resulting graph we again pick random edge weights by assigning integer values between $-5$ and $6$ uniformly distributed and independently for each edge. For the experiments we chose $\ell=10,\ldots,20$ and for each $\ell$ we generated $15$ graphs.

In all above cases, the parameterizations, and in particular the edge weight range, was picked after pilot studies to avoid trivial instances where the optimal solution is the empty graph, all of $G$, or always spanning.

To test whether the odd star inequalities improve the running time of our basic model, we test the model on complete graphs where the edge weights are integer values between $-10$ and $3$, again uniformly and independently distributed for each edge at random. We generate $10$ complete graphs for each number of vertices from $15$ up to $29$.

For each of the considered ILPs (both the basic and the strengthend model), we consider two different variants, one where we separate only on integer solutions, and one where we separate also on fractional solutions. For the sparsified $k$-nearest-neighbor graphs and the $K_n$-cycles we use a time limit of two hours, whereas for the complete graphs the time limit is one hour. The experiments were run on a Xeon E5-2430, 2.5 GHz, under Debian sid, using Gurobi 7.0.2 as the ILP solver, with separation code written in python 3.11.5 using the graphtools library.

\subsection{Experimental experience}

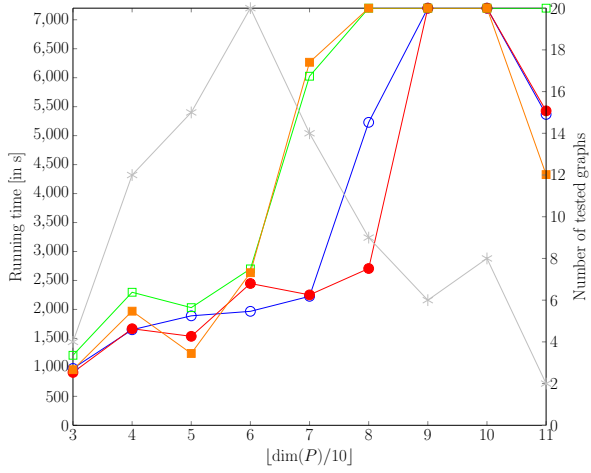
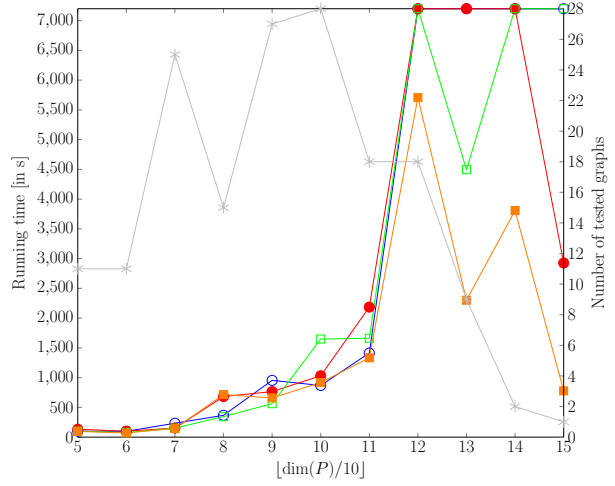
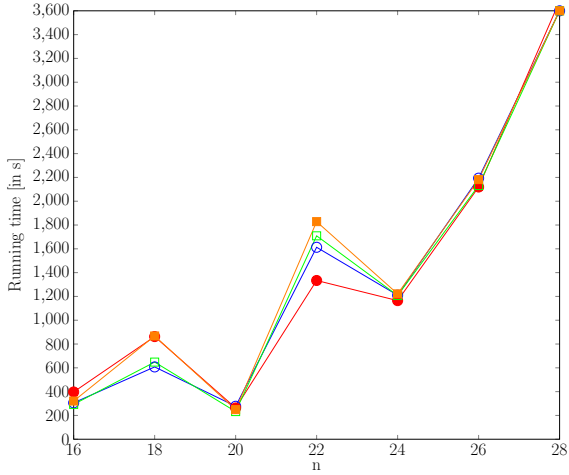
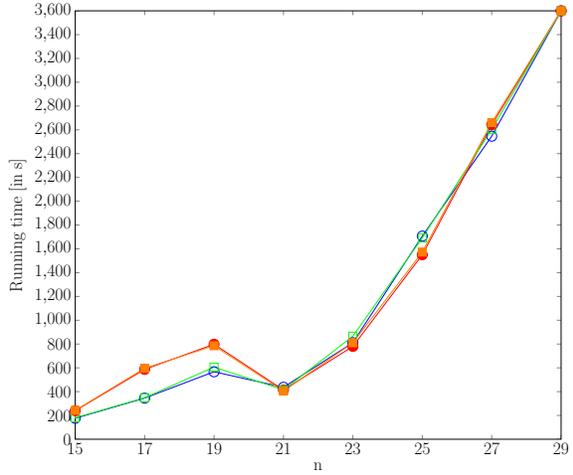
\begin{figure}[p]
    \centering

\begin{subfigure}[b]{\textwidth}
    \centering
    \begin{tabular}{c|c|c}
        & basic model & strengthend model \\ \hline
        only integer separation & \begin{tikzpicture}
            \draw[blue] circle (4pt);
        \end{tikzpicture}  & \begin{tikzpicture}
            \fill[red] circle (4pt);
        \end{tikzpicture} \\ \hline
        also fractional separation & \textcolor{green}{$\square$} & \textcolor{orange}{$\blacksquare$} \\
    \end{tabular}

    \vspace{0.5cm}
\end{subfigure}

\begin{subfigure}[b]{0.49\textwidth}
\centering
    \begin{tikzpicture}[scale=0.37, font=\Large]
 \pgfplotsset{scale only axis,xmin=3, xmax=11, xtick={3,4,5,6,7,8,9,10,11},y axis style/.style={yticklabel style=#1, ylabel style=#1, y axis line style=#1, ytick style=#1}}\begin{axis}[axis y line*=left,ymin=0, ymax=7200,xlabel={$\lfloor\dim(P)/10\rfloor$},ylabel={Running time [in s]},width=17cm,height=15cm]
\addplot[color=blue,mark=o,mark options={scale=2.5}] coordinates{
(3,981.6110424995422) 
(4,1649.3286390304565) 
(5,1889.5907509326935) 
(6,1967.2936289310455) 
(7,2226.237018465996) 
(8,5229.229750871658) 
(9,7200.0) 
(10,7200.0) 
(11,5365.912171006203) 
};
\addplot[color=red,mark=otimes*,mark options={scale=2.5}] coordinates{
(3,912.5063424110413) 
(4,1666.8736114501953) 
(5,1535.807718038559) 
(6,2447.823156952858) 
(7,2251.090502023697) 
(8,2706.396969795227) 
(9,7200.0) 
(10,7200.0) 
(11,5427.442147016525) 
};
\addplot[color=green,mark=square,mark options={scale=2}] coordinates{
(3,1208.3634051084518) 
(4,2296.293756842613) 
(5,2030.8957397937775) 
(6,2699.4376311302185) 
(7,6025.284833431244) 
(8,7200.0) 
(9,7200.0) 
(10,7200.0) 
(11,7200.0) 
};
\addplot[color=orange,mark=square*,mark options={scale=2}] coordinates{
(3,959.8334903717041) 
(4,1969.9949469566345) 
(5,1239.8932600021362) 
(6,2634.024837613106) 
(7,6266.166312932968) 
(8,7200.0) 
(9,7200.0) 
(10,7200.0) 
(11,4327.292586565018) 
};
%\legend{FF, FT, TF, TT}
\end{axis}
\begin{axis}[axis y line*=right,axis x line=none,ymin=0, ymax=20,ylabel={\Large Number of tested graphs}, width=17cm,height=15cm] \addplot[mark=asterisk, mark options={scale=3},color=gray!50!white] coordinates{
(3,4) 
(4,12) 
(5,15) 
(6,20) 
(7,14) 
(8,9) 
(9,6) 
(10,8) 
(11,2) 
}; \end{axis} \end{tikzpicture} 
\caption{Sparsified $k$-nearest-neighbor graphs.}
\end{subfigure}
    \begin{subfigure}[b]{0.49\textwidth}
    \centering
        \begin{tikzpicture}[scale=0.38, font = \Large]
 \pgfplotsset{scale only axis,xmin=5, xmax=15, xtick={5,6,7,8,9,10,11,12,13,14,15},y axis style/.style={yticklabel style=#1, ylabel style=#1, y axis line style=#1, ytick style=#1}}\begin{axis}[axis y line*=left,ymin=0, ymax=7200,xlabel={$\lfloor \dim(P)/10 \rfloor$},ylabel={Running time [in s]},width=17cm,height=15cm]
\addplot[color=blue,mark=o,mark options={scale=2.5}] coordinates{
(5,100.34884405136108) 
(6,97.52113008499146) 
(7,233.54466199874878) 
(8,366.9129011631012) 
(9,954.0645878314972) 
(10,868.4838808774948) 
(11,1413.4851688146591) 
(12,7200.0) 
(13,7200.0) 
(14,7200.0) 
(15,7200.0) 
};
\addplot[color=red,mark=otimes*,mark options={scale=2.5}] coordinates{
(5,133.26454401016235) 
(6,100.15085792541504) 
(7,153.32443189620972) 
(8,681.5879878997803) 
(9,765.1414339542389) 
(10,1030.6407705545425) 
(11,2184.694535970688) 
(12,7200.0) 
(13,7200.0) 
(14,7200.0) 
(15,2926.669333934784) 
};
\addplot[color=green,mark=square,mark options={scale=2}] coordinates{
(5,100.6659688949585) 
(6,75.88252210617065) 
(7,151.0168809890747) 
(8,344.58116698265076) 
(9,561.9778850078583) 
(10,1648.7719041109085) 
(11,1661.6978088617325) 
(12,7200.0) 
(13,4495.392368078232) 
(14,7200.0) 
(15,7200.0) 
};
\addplot[color=orange,mark=square*,mark options={scale=2}] coordinates{
(5,96.56818699836731) 
(6,71.22238183021545) 
(7,145.06829810142517) 
(8,714.5550270080566) 
(9,658.940113067627) 
(10,917.5015355348587) 
(11,1332.3653559684753) 
(12,5706.174446582794) 
(13,2300.689195871353) 
(14,3808.022513985634) 
(15,777.7016031742096) 
};
%\legend{FF, FT, TF, TT}
\end{axis}
\begin{axis}[axis y line*=right,axis x line=none,ymin=0, ymax=28,ylabel={\Large Number of tested graphs}, width=17cm,height=15cm] \addplot[mark=asterisk, mark options={scale=3},color=gray!50!white] coordinates{
(5,11) 
(6,11) 
(7,25) 
(8,15) 
(9,27) 
(10,28) 
(11,18) 
(12,18) 
(13,9) 
(14,2) 
(15,1) 
}; \end{axis} \end{tikzpicture} 
    \caption{$K_n$-cycles.}
    \end{subfigure}

\hspace{1cm}

\begin{subfigure}[b]{0.49\textwidth}
    \centering
\begin{tikzpicture}[scale=0.38, font=\Large]
 \pgfplotsset{scale only axis,xmin=16, xmax=28, xtick={16,18,20,22,24,26,28,30,32},y axis style/.style={yticklabel style=#1, ylabel style=#1, y axis line style=#1, ytick style=#1}}\begin{axis}[%axis y line*=left,
 ymin=0, ymax=3600,xlabel={n},ylabel={Running time [in s]},width=17cm,height=15cm]
\addplot[color=blue,mark=o,mark options={scale=2.5}] coordinates{
(2,3600.0) 
(4,3600.0) 
(6,3600.0) 
(8,3600.0) 
(10,3600.0) 
(12,3600.0) 
(14,3600.0) 
(16,305.99090909957886) 
(18,607.5977159738541) 
(20,276.9848231077194) 
(22,1612.9819685220718) 
(24,1210.2691000699997) 
(26,2193.742634534836) 
(28,3600.0) 
(30,3600.0) 
(32,3600.0) 
};
\addplot[color=red,mark=otimes*,mark options={scale=2.5}] coordinates{
(2,3600.0) 
(4,3600.0) 
(6,3600.0) 
(8,3600.0) 
(10,3600.0) 
(12,3600.0) 
(14,3600.0) 
(16,397.2642389535904) 
(18,862.5308994054794) 
(20,262.70392298698425) 
(22,1334.003366947174) 
(24,1164.7212285995483) 
(26,2119.346780061722) 
(28,3682.1455850601196) 
(30,3600.0) 
(32,3600.0) 
};
\addplot[color=green,mark=square,mark options={scale=2}] coordinates{
(2,3600.0) 
(4,3600.0) 
(6,3600.0) 
(8,3600.0) 
(10,3600.0) 
(12,3600.0) 
(14,3600.0) 
(16,292.96745789051056) 
(18,648.1881295442581) 
(20,232.4930144548416) 
(22,1709.2907639741898) 
(24,1207.521891951561) 
(26,2131.6179959774017) 
(28,3600.0) 
(30,3600.0) 
(32,3600.0) 
};
\addplot[color=orange,mark=square*,mark options={scale=2}] coordinates{
(2,3600.0) 
(4,3600.0) 
(6,3600.0) 
(8,3600.0) 
(10,3600.0) 
(12,3600.0) 
(14,3600.0) 
(16,320.98738861083984) 
(18,867.6261080503464) 
(20,248.60703313350677) 
(22,1828.491159439087) 
(24,1221.5006334781647) 
(26,2181.5736219882965) 
(28,3600.0) 
(30,3600.0) 
(32,3600.0) 
};
%\legend{FF, FT, TF, TT}
\end{axis}

\end{tikzpicture} 
\caption{$K_n$, $n$ even.}
\end{subfigure}
\begin{subfigure}[b]{0.49\textwidth}
\begin{tikzpicture}[scale=0.38, font=\Large]
\pgfplotsset{scale only axis,xmin=15, xmax=29, xtick={15,17,19,21,23,25,27,29},y axis style/.style={yticklabel style=#1, ylabel style=#1, y axis line style=#1, ytick style=#1}}\begin{axis}[%axis y line*=left,
ymin=0, ymax=3600,xlabel={n},ylabel={Running time [in s]},width=17cm,height=15cm]
\addplot[color=blue,mark=o,mark options={scale=2.5}] coordinates{
(1,3600.0) 
(3,3600.0) 
(5,3600.0) 
(7,3600.0) 
(9,3600.0) 
(11,3600.0) 
(13,3600.0) 
(15,175.89331603050232) 
(17,345.62579452991486) 
(19,566.9576070308685) 
(21,436.66851687431335) 
(23,811.2936319112778) 
(25,1706.841698050499) 
(27,2546.4372069835663) 
(29,3600.0) 
(31,3600.0) 
(33,3600.0) 
};
\addplot[color=red,mark=otimes*,mark options={scale=2.5}] coordinates{
(1,3600.0) 
(3,3600.0) 
(5,3600.0) 
(7,3600.0) 
(9,3600.0) 
(11,3600.0) 
(13,3600.0) 
(15,239.50918316841125) 
(17,587.0141520500183) 
(19,798.3033695220947) 
(21,413.74139392375946) 
(23,779.9986544847488) 
(25,1549.9556695222855) 
(27,2642.045290350914) 
(29,3600.0) 
(31,3600.0) 
(33,3600.0) 
};
\addplot[color=green,mark=square,mark options={scale=2}] coordinates{
(1,3600.0) 
(3,3600.0) 
(5,3600.0) 
(7,3600.0) 
(9,3600.0) 
(11,3600.0) 
(13,3600.0) 
(15,179.7943775653839) 
(17,348.53025805950165) 
(19,606.2944670915604) 
(21,415.26660203933716) 
(23,865.8873860836029) 
(25,1691.9803550243378) 
(27,2604.2976915836334) 
(29,3600.0) 
(31,3600.0) 
(33,3600.0) 
};
\addplot[color=orange,mark=square*,mark options={scale=2}] coordinates{
(1,3600.0) 
(3,3600.0) 
(5,3600.0) 
(7,3600.0) 
(9,3600.0) 
(11,3600.0) 
(13,3600.0) 
(15,242.34506905078888) 
(17,596.2224370241165) 
(19,785.6255999803543) 
(21,404.6864905357361) 
(23,810.7244905233383) 
(25,1572.3353599309921) 
(27,2661.3386110067368) 
(29,3600.0) 
(31,3600.0) 
(33,3600.0) 
};
%\legend{FF, FT, TF, TT}
\end{axis}

\end{tikzpicture}
\caption{$K_n$, $n$ odd.}
\end{subfigure}
    \caption{Median running times. We only show data points if more than half of the aggregated instances terminated within the time limit. For the sparsified $k$-nearest-neighbor graphs and the $K_n$-cycles, we took the median of the running time of each graph $G$ with $\lfloor\dim(\TECSP(G)) / 10 \rfloor$. The gray $\ast$-line shows how many different graphs were tested in each data point, where the scale for this is the axis on the right.}
    \label{experimental results}
\end{figure}

An overview of the experimental results can be seen in \Cref{experimental results}. For each class of graphs, we investigate four different algorithms with the options of either using the basic model or the strengthened model and either with or without performing fractional separation.

On the sparsified $k$-nearest-neighbor graphs, allowing fractional separation seems to increase the running time, especially if $60 \le \dim(\TECSP(G)) < 90$. If one allows only integer separation, the strengthend model can be a lot faster as one can see in the data points $\lfloor \dim(\TECSP(G))/10 \rfloor =7,8$, but it is not always beneficial.

%Testing the ILPs on the sparsified $k$-nearest-neighbor graphs reveals that for $\dim(\TECSP(G)) \le 70$, there is no significant difference in running time. For dimension between $70$ and $110$, the best option seems to use fractional separation, but not use the strengthened model with the coparallel class inequalities. In even higher dimensions, fractional separation does not seem worthwhile.

%For the $K_n$-cycles in dimension less than $90$, there is no significant difference in running time. In higher dimension one should allow fractional separation, and for dimension above $100$ additionally use the strengthend model with the coparallel class inequalities.

For the $K_n$-cycles with $\dim(\TECSP(G)) < 80$, there is no significant difference in running time between any variants. If $80 \le \dim(\TECSP(G)) < 90$, the strengthened model has a higher running time than the basic model independent of allowing fractional separation. It seems that the coparallel class inequalities are not worthwhile yet. For $\dim(\TECSP(G)) \ge 90$ the best options seems to allow fractional separation and also use the strengthend model. This especially holds for high dimensions where the other models are often not able to solve the problem before the timeout, while the strengthend model together with fractional separation even achieves running times that are not even close to the timeout.

For complete graphs with less than $20$ vertices, the running time of the strengthened model is higher than the running time of the basic model. This may be due to the fact that on such small instances, the strengthening constraints are not yet worthwhile and only overcrowd the model. For $n \ge 20$ vertices, there is surprisingly no significant difference in the running time for any of the four different algorithm variants.

%We conclude that if there is a graph with large coparallel classes and high dimension of the corresponding $2$-edge-connected subgraph polytope, it can pay off to use the strengthened model with the coparallel class inequalities.

We conclude that it is hard to say which model is the best one. In some cases, it can pay off to use the strengthened model in combination with fractional separation. In other cases, fractional separation seems to increase the running time and in these cases it is hard to tell from the data if the strengthened model significantly improves the running time.

\section{Conclusion}
\label{conclusion}

We have studied different supporting hyperplanes for the $2$-edge-connected subgraph polytope when the underlying graph $G$ is $2$-edge-connected. The asymmetric cut inequalities and the connectivity cut inequalities together with the trivial box inequalities describe all lattice points contained in $\TECSP(G)$. Therefore, we get an ILP for finding a best weighted $2$-edge-connected subgraph of a $2$-edge-connected graph $G$, which directly yields an algorithm to find a best weighted $2$-edge-connected subgraph of an arbitrary graph $G'$. We also studied other supporting hyperplanes, namely, the coparallel class inequalities and the odd star inequalities, where the latter are only defined for complete graphs. We experimentally confirmed that the coparallel class inequalities may improve the running time of the so far considered ILP on certain classes of graphs. On the other hand, the odd star inequalities do not seem to make a difference in running time. Some open questions remain. Of course, the most interesting but probably hard to answer question is

\begin{question}
    What is a complete facet-description of $\TECSP(G)$ for any $2$-edge-connected graph $G$?
\end{question}

Therefore, some easier questions would be the following two.

\begin{question}
    What are other classes of supporting hyperplanes for the $\TECSP(G)$, and when are they facet-defining?
\end{question}

\begin{question}
    For which graphs $G$ is $\TECSP(G)$ completely described by the so far studied inequalities?
\end{question}

\bibliographystyle{alpha}
\bibliography{bibliography.bib}

% Auch noch Theoreme ergänzen für komponentenweise 2-Kanten-zusammenhängend als Bemerkung bzw. Section? Die gültigen Ungleichungen können ja nur größer gleiche Dimension haben und bleiben damit Facetten.

\section{Appendix}

\label{appendix}

Here, we will explain why the correction in \cite[Corollary 4.23]{BARAHONA198640} is necessary. We use some standard matroid terminology, which we will not define here; see \cite{OxleyMatroids} for a reference.

The original version of \cite[Corollary 4.23]{BARAHONA198640} states the following.

\begin{theorem}[\cite{BARAHONA198640}]
    Let $G$ be a graph, let $E_3 \subseteq E(G)$ be the edges not contained in a cut of size at most $3$ and let $E' \subseteq E(G)$ be a maximal set of edges not containing bridges or $2$-cuts. Then $\ESP(G)$ is given by
    \begin{enumerate}[(a)]
        \item $x_e=0$ for each bridge $e \in E(G)$,
        \item $x_e-x_f=0$ for each minimal cut $\{e,f\}$ of size two,
        \item $0 \le x_e \le 1$ for each $e \in E(G) \setminus E_3$ that is not a bridge,
        \item $x(F)-x(C\setminus F) \le \vert F \vert -1$ for each minimal cut $C \subseteq E'$ without chord, and each $F \subseteq C$, $\vert F \vert$ odd.
    \end{enumerate}
    Moreover, the system above is nonredundant.
\end{theorem}

%\textcolor{red}{Wir haben auch Bedingung dass $\vert C \vert \ge 3$ ist hinzugefügt.}

The changes we made in \Cref{facetsESP} are to restrict the condition in \textit{(b)} to less cases, define $E_3$ to be the edges that are contained in cuts of size $3$, and we modified the condition in \textit{(d)}.

We made the first correction, because otherwise the system of linear (in-)equalities is not nonredundant in general, since $x_{e_1}-x_{e_i}=0$ for each coparallel class $C=\{e_1,\ldots,e_k\} \in \CP(G)$ and all $i=2,\ldots,k$ already implies $x_e-x_f=0$ for each minimal cut $\{e,f\}$ of size two.

In the following, let $\M=(E,\C)$ denote a binary matroid with ground set $E$ where the elements in $\C \subseteq 2^E$ are the circuits of $\M$. For a graph $G$, let $\M(G)$ denote its corresponding graphic matroid. In \cite{BARAHONA198640}, the \emph{cycle polytope} of a binary matroid is studied, where the cycle polytope of a matroid $\M$, denoted by $\Cyc(\M)$ is given by
\begin{align*}
\Cyc(\M)\coloneqq \conv\{\chi^C \mid C \subseteq E \text{ is a (matroid-)cycle}\}.
\end{align*}
Here, the incidence vector $\chi^C$ is defined analogously as in \eqref{incidence vector} and a \emph{(matroid-)cycle} is the disjoint union of circuits of a matroid. Observe that this is contrary to standard graph theory, where --- if $\M=\M(G)$ is a graphic matroid --- the elements of $\C$ are the edge sets of (simple) cycles of $G$, whereas the edge-disjoint union of cycles has no well-established name. For a graphic matroid $\M(G)$ we have
\begin{align*}
    \Cyc(\M(G))=\ESP(G),
\end{align*}
and since graphic matroids are always binary, all the results of \cite{BARAHONA198640} concerning $\Cyc(\M)$ also apply to $\ESP(G)$.

Two elements in a binary matroid $\M=(E,\C)$ are \emph{coparallel}, if they are parallel in the dual matroid, i.e., they form a cocircuit of size $2$. A \emph{coparallel class} $C \subseteq E$ is a maximal subset of the groundset such that every pair of distinct elements of $C$ are coparallel. Observe, that for a graphic matroid these definitions coincide with our definition of two coparallel elements and a coparallel class, resp., in a graph. For a matroid $\M$, let $\CP(\M)$ denote the set of coparallel classes of $\M$.

For a binary matroid $\M=(E,\C)$, as in \cite{BARAHONA198640}, we will denote by $\overline{\M}$ the matroid that is obtained by deleting every coloop of $\M$ and, for every coparallel class $\{e_1,\ldots,e_k\} \in \CP(\M)$, contracting the elements $e_2,\ldots,e_k$. Hence, $\overline{\M}$ is a matroid without coloops and without coparallel elements. In \cite{BARAHONA198640}, it has been shown that each facet-defining inequality of $\Cyc(\overline{\M})$ is also a facet-defining inequality of $\Cyc(\M)$. Note that this is well-defined, since the ground set of $\overline{\M}$ is a subset of $E$.

Let $E' \subseteq E$ be all elements of $\M$ that are not a coloop and let $\pi\colon E' \to E'$ be the map that sends each element of $E'$ to the unique element that is contained in the same coparallel class as $e$ and has not been contracted in $\overline{\M}$. Note that for every cocircuit $C$ and coparallel elements $e,f \in E$ of $\M$, also $C \triangle \{e,f\}$ is a cocircuit.\footnote{Since $\M$ is binary.} Thus, for every cocircuit $C$ of $\M$, with $\vert C \vert \ge 3$, $\pi(C)$ is also a cocircuit in $\M$ and in $\overline{\M}$ with $\vert C \vert=\vert \pi(C) \vert$.\footnote{Since $C$ is a cocircuit with at least three elements, no two elements contained in $C$ can be coparallel.}

We can now comment on the second correction we made in \Cref{facetsESP}. In the original version of \cite[Corollary 4.23]{BARAHONA198640}, for a graph $G$, the edges that are not contained in a cut of size at most $3$ are denoted by $E_3$. Then it is said that $0 \le x_e \le 1$ is part of the nonredundant inequality description of $\ESP(G)$, i.e., the inequalities define facets of $\ESP(G)$, if and only if $e \in E \setminus E_3$. This does not fit with the following result, which is \cite[Theorem 4.8, Corollary 4.9]{BARAHONA198640}.

\begin{theorem}[\cite{BARAHONA198640}]
    Let $\M'=(E,\C)$ be a binary matroid without coloops and without coparallel elements. If $e \in E$ does not belong to a cotriangle, i.e., a cocircuit with three elements, then the inequalities
    \begin{align*}
        0 \le x_e \le 1
    \end{align*}
    define facets for $\Cyc(\M')$.
\end{theorem}
Furthermore, after the proof of \cite[Corollary 4.9]{BARAHONA198640}, it is shown that if $e \in E$ is contained in a cotriangle, the inequalities $0 \le x_e \le 1$ do not define facets of $\Cyc(\M)$, where here $\M$ is an arbitrary binary matroid.

Applying all of the above to an arbitrary binary matroid $\M=(E,\C)$, we see that if an element $f$, that is not a coloop, is not contained in a cocircuit of size $3$, also $\pi(f)$ is not contained in a cocircuit of size $3$ in $\overline{\M}$, and hence, the inequalities
\begin{align*}
    0 \le x_f \le 1
\end{align*}
define facets of $\Cyc(\M)$ if and only if $f$ is not contained in a cocircuit of size $3$, which explains the second correction.

Let us now comment on the third correction we made in \Cref{facetsESP}. In \cite[Corollary 4.23]{BARAHONA198640}, it is stated that for a minimal cut $C \subseteq E(G)$ of a graph $G$ and $F \subseteq C$, $\vert F \vert$ odd, the inequality
\begin{align*}
    x(F)-x(C \setminus F) \le \vert F \vert -1
\end{align*}
defines a facet if and only if $C$ has no chord. This is a corollary from \cite[Theorem 4.20]{BARAHONA198640}, which is the following.

\begin{theorem}
    \label{Theorem 4.20}
    If the binary matroid $\M$ has no $F_7^\ast$ minor, and $C=\{e_1,e_2,\ldots,e_k\}$, $k \ge 3$, is a cocircuit without chord, then the inequality
    \begin{align*}
        x_{e_1}-x(\{e_2,\ldots,e_k\}) \le 0
    \end{align*}
    defines a facet of $\Cyc(\M)$.
\end{theorem}

Let $C=\{e_1,e_2,\ldots,e_k\} \in \C$ be a cocircuit of a binary matroid $\M=(E,\C)$ with $\vert C \vert \ge 3$. Then $\pi(C)$ is a cocircuit of $\overline{\M}$ with $\vert \pi(C) \vert =\vert C \vert \ge 3$, and if $\pi(C)$ has no chord in $\overline{\M}$, this implies that
\begin{align*}
    x_{e_1}-x(\{e_2,\ldots,e_k\}) \le 0
\end{align*}
defines a facet of $\Cyc(\M)$ due to \Cref{Theorem 4.20}. This actually strengthens \Cref{Theorem 4.20} since if a cocircuit $C$ has no chord in $\M$, then it neither has a chord in $\overline{\M}$.

In \cite{BARAHONA198640}, it is also claimed that for a cocircuit $C=\{e_1,\ldots,e_k\}$, $k \ge 3$, of a binary matroid $\M=(E,\C)$ that has a chord, the inequality
\begin{align*}
    x_{e_1}-x(\{e_2,\ldots,e_k\})\le 0
\end{align*}
does not define a facet of $\Cyc(\M)$. However, this is not true for all binary matroids, for example, consider the graphic matroid $\M(G)$ of the graph $G$ shown in \Cref{counterexample}. The cocircuit $\{e_1,e_2,e_3\}$ has the chord $e_4$, since $\{e_1,e_2,e_3\}=\{e_1,e_2,e_4\} \triangle \{e_3,e_4\}$ and $\{e_4\}=\{e_1,e_2,e_4\} \cap \{e_3,e_4\}$, but the inequality
\begin{align*}
    x_{e_1}-x_{e_2}-x_{e_3} \le 0
\end{align*}
defines a facet of $\Cyc(\M(G))$ as one can easily verify. The argument used in \cite{BARAHONA198640} is the following. Suppose the cocircuit $C$ has a chord, i.e., there are cocircuits $C_1,C_2$ in $\M$, such that $C_1 \triangle C_2=C$ and $C_1 \cap C_2 =\{h \}$ for an $h \in E$. Without loss of generality we assume that $e_1 \in C_1$, then the inequality
\begin{align*}
    x_{e_1}-x(\{e_2,\ldots,e_k\}) \le 0
\end{align*}
is the sum of the two inequalities
\begin{align*}
    x_{e_1}-x(C_1 \setminus \{e_1\}) \le 0 \text{, and } x_{h}-x(C_2 \setminus \{h\}) \le 0,
\end{align*}
and therefore does not define a facet. However, this is not true, if one of the cocircuits has only two elements and for the other one the corresponding inequality defines a facet. This can happen as we have just seen in the example above. The correct statement would be that if the cocircuit $\pi(C)$ has no chord in $\overline{\M}$, then the considered inequality does not define a facet. In this case, the argument above works since there are no cocircuits of size two in $\overline{\M}$. Hence, a stronger version of \cite[Theorem 4.20]{BARAHONA198640} is the following.

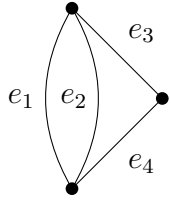
\begin{figure}
    \centering
        \begin{tikzpicture}[scale=1.2]
        \coordinate (A) at (0,0);
        \coordinate (B) at (0,2);
        \coordinate (C) at (1,1);

        \fill (A) circle (2pt);
        \fill (B) circle (2pt);
        \fill (C) circle (2pt);

        \draw(A) to [bend left] node[midway, left] {$e_1$} (B);
        \draw(A) to [bend right] node[midway, left] {$e_2$} (B);
        \draw (B) to node[midway, above right] {$e_3$} (C) ;
        \draw (A) to node[midway, below right] {$e_4$} (C);
        \end{tikzpicture}
    \caption{The graph $G$.}
    \label{counterexample}
\end{figure}

\begin{theorem}
    If the binary matroid $\M$ has no $F_7^\ast$ minor, and $C=\{e_1,e_2,\ldots,e_k\}$, $k \ge 3$, is a cocircuit, then the inquality
    \begin{align*}
        x_{e_1}-x(\{e_2,\ldots,e_k\}) \le 0
    \end{align*}
    defines a facet if and only if $\pi(C)$ has no chord in $\overline{\M}$.
\end{theorem}
This explains our third correction of \cite[Corollary 4.23]{BARAHONA198640} we made in \Cref{facetsESP}.

\end{document}